\documentclass[12pt,a4paper]{article}
\usepackage{amsmath,amsfonts,amssymb,amsthm,amscd,mathrsfs}
\usepackage{stmaryrd}
\usepackage{cases}
\usepackage[english]{babel}
\usepackage[T1]{fontenc}
\usepackage{aeguill}
\usepackage{enumerate}

\usepackage{graphics,epsfig,color,cite,bbm}

\definecolor{light}{gray}{.9}

\newcommand{\mb}[1]{\mathbb{#1}}

\newcommand{\mc}[1]{\mathcal{#1}}
\newcommand{\mbf}[1]{\mathbf{#1}}

\newcommand{\im}{\text{i}}

\newcommand{\lp}{\langle}
\newcommand{\rp}{\rangle}
\newcommand{\ra}{\rightarrow}
\newcommand{\ve}{\varepsilon}

\newcommand{\vp}{\varphi}

\DeclareMathOperator{\spec}{Spec}

\def\variance{{\rm Var}}
\def\ent{{\rm Ent\,}}

\DeclareMathOperator{\supp}{supp}
\DeclareMathOperator{\tr}{Tr}

\DeclareMathOperator{\hess}{Hess}

\DeclareMathOperator{\id}{Id}

\DeclareMathOperator{\range}{Ran}

\def\Hess{{\rm Hess\,}}

\def\supp{\mathop{\rm supp} \nolimits} 

\def\qed{\hfill\hbox {\hskip 1pt \vrule width 4pt height 6pt depth 1.5pt
\hskip 1pt}\\} 

\def\and {{\rm \; and \;}}

\def \tr{{\rm \, Tr\;}}
\def\dim {{\rm \; dim  \;}}

\makeatletter

\@addtoreset{equation}{section}
\makeatother

\newtheorem{theorem}{Theorem}[section]
\newtheorem{lemma}[theorem]{Lemma}

\newtheorem{proposition}[theorem]{Proposition}
\newtheorem{definition}[theorem]{Definition}
\newtheorem{remark}[theorem]{Remark}
\newtheorem{corollary}[theorem]{Corollary}

\newtheorem{theorem*}{Theorem}

\begin{document}

\title{Small noise spectral gap asymptotics for a large system of nonlinear diffusions}

\author{Giacomo Di Ges\`{u}\thanks{CERMICS -- \'Ecole des Ponts ParisTech,
6-8 avenue Blaise Pascal, F-77455 Marne-la-Vall\'ee Cedex 2.  (giacomo.di-gesu@enpc.fr)}\ ,\  
Dorian~Le~Peutrec\thanks{D\'epartement de Math\'ematiques,
UMR-CNRS 8628, B\^at. 425, Universit\'e Paris~Sud, F-91405 Orsay
Cedex. 
 (dorian.lepeutrec@math.u-psud.fr)}
}

\maketitle

 \begin{abstract}
We study the $L^2$ spectral gap of 
a large system of strongly coupled diffusions on unbounded state space and subject to a  double-well potential. 
This system can be seen 
as a spatially discrete approximation 
of the stochastic Allen-Cahn equation
on the one-dimensional torus. We prove upper and lower bounds for 
the leading term of the spectral gap in the small temperature regime
with uniform control in the system size. 
The upper bound is given by an Eyring-Kramers-type formula. The lower bound is proven to hold also for the
logarithmic Sobolev constant. We establish a sufficient condition for the asymptotic optimality of the upper bound
and show that this condition is fulfilled
under suitable assumptions on the growth of the system size.
Our results can be reformulated in terms of a semiclassical Witten Laplacian in large 
dimension.  
\end{abstract}

\


\section{Introduction }

This paper  concerns
the rate of convergence to equilibrium at  low temperature  of 
a stochastic interacting particle system, 
which may be described as follows.
There are $N$ particles, at each time  $t\geq 0$ the state of the $k$-th particle is 
a real random number $\xi_k(t)$ and 
the trajectory $\xi_k = (\xi_k(t))_{t\geq 0}$ satisfies
for some fixed $\mu>1$ the  stochastic differential equation
\begin{equation}  \label{sde}   d\xi_k   \   =  \         \   \Big[  \   \ 
  \mu   \  \frac{   \xi_{k+1}  +   \xi_{k-1}
- 2\xi_k }{4\sin^2{\frac{\pi}{N}}}      \   +   \     \xi_k   \  -   \
\xi_k^3      \    \   \Big]   \   dt    \   +   \   
\sqrt {2h N}  \   dB_k  \      \    .    
\end{equation}
Here  $B_1= (B_1(t))_{t\geq 0}, \dots, B_N = (B_N(t))_{t\geq 0}$
are $N$ independent standard Brownian motions,
$h$ is a positive constant and $\xi_{N+1}:= \xi_1$, i.e. periodic
boundary conditions are assumed.    
When $h>0$ is kept fixed and $N$ is large, 
system \eqref{sde} can be 
seen as a discrete space approximation of
the stochastically perturbed Allen-Cahn equation on the interval
 $(0,\frac{2\pi}{\sqrt \mu})$:
 \begin{equation}  \label{spde}   
d u (x,t)   \   =  \         \   \big[  \    
      \partial^{2}_x u  (x,t)           \   +   \     u(x,t)
      \  -      \    u^3(x,t)   \       \big]   \   dt    \   +   \   
\sqrt {2h}  \   dW (x,t)\      \    ,    
\end{equation}
where now $(x,t)\in (0,\frac{2\pi}{\sqrt \mu})
\times (0,\infty)$,  the boundary condition $u(0,t) =   u(\frac{2\pi}{\sqrt \mu}, t)$ has to be satisfied
for every $t\geq 0$, and $dW$ is a space-time white noise. Thus,
for large $N$, 
one might think   of  $   \xi_k(t)     \sim      u \big( \frac{k}{N }   \frac{2\pi}{ \sqrt\mu}, t \big)   $, 
and of the chain $\xi(t) = (\xi_1(t), 
\dots, \xi_N(t))$ as giving the position
at time $t$ of an 
elastic ring of length $\frac{2\pi}{\sqrt \mu}$ moving in a highly viscous, noisy environment
and subject to a simple bistable external force. 
\\[0.1cm]
Equation~\eqref{spde} is a basic and widely studied stochastic partial differential 
equation, see e.g. \cite{FaJL, Fun,BrDmPr,GoMa,KORV,Hai2,BeGe,OtWeWe,DaZa,Ba}
and references therein.
For a more general background on the particle system~\eqref{sde} we refer
to~\cite{BeFeGe1,
BeFeGe2}. See also~\cite{BaBoMe} for aspects closely related to this work. The convergence of~\eqref{sde} to~\eqref{spde} for $N\ra \infty$
is discussed in~\cite{Ba}. 

\

\noindent{\bf Relaxation properties: heuristics and previous results.}
\\[0.1cm]
\noindent
For each fixed $h>0$
and number of particles $N$, the long time behaviour of~\eqref{sde} is described  
by its unique equilibrium distribution, explicitly given by the probability measure
on $\mb R^N$
 \[       m_{h,N}(dx)        \   :  =    \        \frac{  e^{-\frac{V(x)}{hN}} \ dx}{\int_{\mb R^N}    e^{-\frac{V(x)}{hN}}   dx}         \   ,            \] 
where the energy function $V:\mb R^N  \ra  \mb R$ is defined
as
\begin{equation}\label{defV}   V(x)     \   :=    \       \sum_{k=1}^N    \big(\    \frac 14 x_{k}^4   \  -   \  \frac 12 x^2_{k} \    \big)    \   +     \      \mu   \   \sum_{k=1}^N  
   \frac { (x_{k} - x_{k+1})^2}{8\sin^2(\frac{\pi}{N})}   \   +     \frac N4         \   ,     
   \end{equation}
with $x_{N+1} :=x_1$. This follows from the observation that the drift term 
in~\eqref{sde} is the gradient of $V$ and  from general facts about gradient-type diffusions. 
Similarly, 
for fixed $h>0$, there exists a
unique equilibrium distribution $m_{h,\infty}$ for
the infinite-dimensional system~\eqref{spde}, see \cite{DaZa,ReVe}. One might say that 
at equilibrium no ``phase transition'' occurs in the thermodynamic limit $N\ra \infty$. 
On the contrary, since for each $N$ the energy  $V$ admits  two local minima, 
given by 
\[     I_+   \  :=   \   (1,\dots, 1)    \  \  \text{ and }     \   \    I_-    \  :=  \
 (-1,\dots, -1)   \  ,     \]
 the deterministic dynamics
$ d\xi     =     -    \nabla V(\xi)     dt  $,  obtained from \eqref{sde} by setting $h=0$, admits two stable equilibrium points. Thus, when $h$ is positive but small, the typical picture of a so-called metastable dynamics emerges
\cite{FrWe, FaJL, BaBoMe}:
the system quickly reaches a local equilibrium 
in the basin of attraction of $I_+$ or $I_-$, 
depending on its initial condition; this local equilibrium endures 
for a long time, since, in order to be able to explore the whole state space
and distribute according to the global equilibrium $m_{h,N}$, 
the system has to wait for a sufficiently large stochastic fluctuation allowing to overcome the energetic barrier separating $I_+$ and $I_-$. 
The critical time scale at which such transitions between minima typically occur is exponentially large in the parameter $h$. Thus, for $h\ra 0$, 
one observes a significant slowdown in the relaxation towards $m_{h,N}$.

\
 
\noindent
The aim of this paper is to quantify the mentioned slowdown 
in the approach to equilibrium of~\eqref{sde}
when at the same time $h$ is small and $N$ is large. More specifically 
we shall study for $h\ra0$ and $N\ra \infty$ the behaviour of the
Poincar\'{e} constant $\lambda(h,N)$ and the logarithmic Sobolev constant 
$\rho(h,N)$ of~\eqref{sde}. 
These are defined as the largest constants satisfying respectively,
for every $\vp\in H^1(\mb R^N,m_{h,N})$,
the weighted Poincar\'{e} inequality 
\begin{equation}\label{specgapdef}        
{\lambda (h,N)}    \   \variance_{m_{h,N}}  (\vp)    \      \leq   \  
h N        \int  |\nabla \vp|^2   \       d m_{h,N}     \     ,       
\end{equation}
and the Gross inequality (or logarithmic Sobolev inequality)
\begin{equation}\label{logsobdef}            \rho(h,N)   \    \ent_{m_{h,N}}(\vp^2)      \      \leq   \   2    \,
h N           \int  |\nabla \vp|^2   \      d m_{h,N}     \     .        
\end{equation}
Here $\variance_{m_{h,N}}$ and $\ent_{m_{h,N}}$ denote the variance 
and entropy with respect to $m_{h,N}$, i.e.
$\variance_{m_{h,N} }(\vp)    :=    \int \vp^2   d m_{h,N}    -     \big( \int \vp  \    d m_{h,N}  \big)^2    $ and, for  $\vp\geq 0$, $\ent_{m_{h,N} }(\vp) :=   \int    \vp  \ \log \vp   \   d m_{h,N}    -      \int \vp  \   d m_{h,N} \log \big(\int   \vp  \  d m_{h,N}\big) $.  
The right hand side of~\eqref{specgapdef} is also called the Dirichlet form associated with the Markov process
defined by~\eqref{sde}. 
\\[0.1cm]
It is well-known that the Poincar\'{e} constant and 
logarithmic Sobolev constant  give the exponential rate of convergence to equilibrium,
respectively in variance and in entropy. We refer e.g. to Theorem 4.2.5 and 5.2.1 in~\cite{BaGeLe}, which also 
gives a general overview of the interplay between functional inequalities and Markov processes. We stress that, from the point of view of spin systems in statistical
mechanics,  we are dealing here with the problem of relaxation to equilibrium in a case of continuous unbounded 
single-spin state space and nonconvex energy function (see e.g.
  \cite{Le, Zeg, BoHe,BoHe2} in this context).  Concerning
 exponential convergence of stochastic equations in infinite dimensions with fixed noise parameter $h$
 we point e.g. to~\cite{GoMa,Hai1,Hai2,DaZa}.

 \

\noindent
If $N$ is kept fixed it is known  that the leading asymptotic behaviour of 
$\lambda(h,N)$ in the limit $h\ra0$ is given by an Eyring-Kramers-type formula
(see~\cite{BovEckGayKl,BovGayKl,HKN}, treating generic multiwell-diffusions in the small noise regime, and also the recent \cite{MeSch,Mic2}). 
More specifically, it follows for example from~\cite{HKN} and some straightforward adaptations of their arguments, 
that 
 \begin{equation}\label{fixeddimension}
\lambda(h,N)
\  =    \ 
    \frac 1\pi   \  \left|\frac{\det\Hess V(I_-)}{\det\Hess V(0)}\right|^{\frac12}    \  
e^{-\frac{1}{4h}}
\  \big(  \   1   \  +    \     \epsilon(h,N)    \   \big)        \    ,
\end{equation}
where the error $\epsilon(h,N) $ satisfies, for $h>0$ sufficiently small,
$  |\epsilon(h,N)|     \   \leq    \      C_N  \   h  $. Here  
$C_N$ is some positive constant which  
may a priori explode in $N$. On the other hand, as was already observed in~\cite{St}, the prefactor 
in~\eqref{fixeddimension} is convergent in the limit $N\ra \infty$: 
\begin{eqnarray}
\label{prefactor}
p(N)   \ :=    \ \frac 1\pi   \  \left|\frac{\det\Hess V(I_-)}{\det\Hess V(0)}\right|^{\frac12}\ 
\underset{N\to+\infty}{\longrightarrow}\ 
  \frac{ \sinh (\pi \sqrt{2\mu^{-1}})}{ \pi  \sin (\pi \sqrt{\mu^{-1}})}    \    .   
 \end{eqnarray}
Similarly, regarding the log-Sobolev constant $\rho(h,N)$, it follows again from general results (see\cite{MeSch}), that for fixed $N$
the leading term of $\rho(h,N)$ is again given by $p(N)e^{-\frac{1}{4h}}$.
We stress that also here, as for the error
in~\eqref{fixeddimension}, there is no control in $N$
on the error term. Thus no rigorous conclusion in the limit $N\ra \infty$
can be directly inferred  from these results.\\  

\noindent
On the other hand, 
rather strong results have been obtained in the analysis of 
the mean time needed for system~\eqref{sde} to go from $I_+$
to $I_-$: indeed it has been shown that 
an Eyring-Kramers-type
formula holds for  this transition time, with an error which is uniform 
in  $N$ (see in particular~\cite{BaBoMe} and \cite{Ba, BeGe},
which extend the results to the infinite-dimensional system~\eqref{spde}
and even to more general situations). Nevertheless, while
 the asymptotic relation
between stochastically defined mean transition times and 
analytic objects as $\lambda(h,N)$ is well-established
in very general situations for fixed $N$ (see again~\cite{BovGayKl}),
to the best of our knowledge
there are no rigorous results on how it might behave
in the regime of large $N$, even in the specific  model we are considering. 
In this paper we do not rely on the mentioned results on mean transition times and 
rather use purely analytical arguments, partly inspired by the semiclassical spectral-theoretic approach developped in \cite{HKN}. 

\

\noindent
{\bf Statement of the main results}
\\[0.1cm]
\noindent
Our first main result below shows that the Eyring-Kramers
 formula~\eqref{fixeddimension} provides an upper bound on $\lambda(h,N)$ with an error term which can indeed be  uniformly controlled in the system size $N$. Moreover it provides a quantitative lower bound 
 at logarithmic scale
on $\rho(h,N)$ which is independent of $N$. In particular it 
ensures that $\rho(h,N)$ and $\lambda(h,N)$ 
do not degenerate for any fixed $h$.
One might say that 
no ``dynamical phase transition'' occurs in the thermodynamic limit $N\ra \infty$ (see also~\cite{GoMa}).

\begin{theorem}\label{th.rough}
For every $\delta>0$ there exists a constant $C_\delta>0$ such that 
for every $h>0$ and every $N\in \mb N$
\[   C_{\delta}  
   \ e^{-\frac{3+2\sqrt 2+\delta}{24h}}   \  e^{-\frac {1}{4h}}
      \    \leq    \   \rho(h,N)   \   \leq    \  \lambda(h, N)    \   \leq      \  p (N)    \ e^{-\frac {1}{4h}} 
\    \big(   \   1   \  +   \   \epsilon(h,N)   \  \big)      \       ,        \]
where the prefactor $p(N)$ is given by \eqref{prefactor} and   the error term $\epsilon(h,N) $ satisfies 
\[    \exists C>0   \  \text{ s.t. }  \  \forall h\in (0,1]  \   ,   \    \forall N\in \mb N     \   ,     \   \   \    \       |\epsilon(h,N)|     \   \leq    \      C  \   h          \  .   \]
\end{theorem}
\noindent
The exponential decay in $h$ given by the lower bound in Theorem~\ref{th.rough} appears to be rather rough, but 
unfortunately, when insisting to get bounds with uniform control in $N$, 
it is for the moment not clear how one could obtain 
a substantial improvement, even when focusing only on $\lambda(h,N)$. 
For the latter one can exploit the spectral theory of self-adjoint operators:
the generator of the Markovian semigroup giving the evolution
of \eqref{sde} is indeed the differential operator
\begin{equation}\label{generator}
L_{h}    \   :=    \      -   h N \Delta     \  +    \     \nabla V \cdot \nabla       \     .  
\end{equation}
The closure in $L^2(m_{h,N})$ of $L_{h}$
acting on $C_c^\infty(\mb R^N)$, which we still denote by $L_h$, is self-adjoint and nonnegative,  admits $0$ as eigenvalue and has purely discrete spectrum for each $h,N$ fixed
(see Section~\ref{subse.Witten} for more details). As a consequence of the Max-Min principle, its spectral gap, defined as its first nonzero eigenvalue,
coincides with $\lambda_{h,N}$.

\

\noindent
According to our second main result below, the problem 
of obtaining the Eyring-Kramers formula
as lower bound for $\lambda(h,N)$    
can then be reduced 
to the problem of proving a suitable separation between $\lambda(h,N)$
and the next eigenvalue of $L_h$.
More precisley, the existence of a uniform lower bound on the ``second spectral gap'' in a certain regime in which $N$ possibly grows 
to infinity, turns out to be sufficient for the validity
of the Eyring-Kramers formula  in the same regime:

\begin{theorem}
 \label{th.main2}
 Assume there exist constants $h_0, \delta>0$ and,  for each $h\in (0,h_0]$, 
 a set $\mc N(h)\subset \mb N $
  such that  
\begin{equation}\label{conditionsecondsg}    \forall h\in(0,h_0]  \   ,   \    \forall N\in \mc N(h)     \   ,     \   \   \     \
 \spec(L_h)   \   \cap  \   ]\lambda(h,N), \lambda(h,N) + \delta[         \       =       \   \emptyset     \    .   
 \end{equation}
 Then 
 \begin{equation*}
\label{eq.main}
\lambda(h,N)
\  =    \ 
    p(N)    \  
e^{-\frac{1}{4h}}
\  \big(  \   1   \  +    \     \epsilon(h,N)    \   \big)        \    ,
\end{equation*}
where  the prefactor $p(N)$ is given by   \eqref{prefactor} and the error
term $\epsilon(h,N) $ satisfies 
\[      \exists C>0    \    \text{ s.t. }    \ 
 \forall h\in (0,h_0]  \   ,   \    \forall N\in \mc N(h)     \   ,     \   \   \     \    |\epsilon(h,N)|     \   \leq    \      C  \   h      \  .   \]
\end{theorem}

\noindent
Our last main theorem implies that 
there exist regimes with unbounded $N$ 
under which the Eyring-Kramers formula~\eqref{fixeddimension} holds  
with bounded error $\epsilon(h,N)$. Indeed, in order
to be in the situation of 
Theorem~\ref{th.main2},
it is enough that 
$N$ grows slower than $h^{-\frac 34}$: 
\begin{theorem}
 \label{th.main3}
 Let $C>0$ and $\alpha \in (0, \frac 34)$.  Then there exist constants $h_0,\delta>0$ 
 such that condition \eqref{conditionsecondsg} in Theorem \ref{th.main2} is fulfilled with
 \[     \mc N(h)    \   =   \    \big\{   \   N\in \mb N   :       N\leq C h^{-\alpha} \    \big\}    \   .    \]
 \end{theorem}

\noindent 
The above results concerning $\lambda_{h,N}$ can
be equivalently reformulated in terms of 
splitting properties of the ground state of a specific semiclassical 
Schr\"{o}dinger operator in large dimension. 
This is a consequence of the well-known ground state transformation, 
see e.g. \cite{JLMaSc}: up to conjugation with $e^{-\frac{V}{2hN}}$ and some $N$-dependent dilatation
(see Subsection~\ref{subse.Witten}
for more details), the operator $h\,L_h$ 
turns out to be unitarily equivalent to the operator acting in the flat space $L^2(dx)$ and defined through 
\begin{equation}
\label{eq.defDelta0}
\Delta^{(0)}_{f,h}    \    :=    \       -   h^2 \Delta     \     +    \ 
 |\nabla f|^2      \  -    \       h    \Delta f        \   , \quad\text{where}
 \quad  f(x)    \    : =   \      \frac {V(\sqrt N x  )}{2N} \ .
\end{equation}
We like to mention that a semiclassical Schr\"{o}dinger operator having the form of $\Delta^{(0)}_{f,h}$, 
with $f$ a generic smooth function, is also called semiclassical Witten Laplacian associated to $f$. The superscript $(0)$ stresses that we consider only 
operators on functions, 
the Witten Laplacian being more generally defined 
through a supersymmetric extension on the full algebra of differential forms. 
The operator acting on $p$-forms is  commonly denoted by $\Delta^{(p)}_{f,h}$
and connects in the semiclassical limit $h\ra0$ topological properties of the underlying manifold to the topology of the energy landscape induced by $f$
\cite{Wit, HeSj4, CFKS}.

\

\noindent
We stress that, even if one focuses only on the operator $\Delta^{(0)}_{f,h}$
acting on functions (that is, equivalently on the diffusion operator $L_h$, 
as in the present paper), 
the enlarged supersymmetric point of view may provide further insights and
a powerful technical tool. We refer
especially to~\cite{Sjo, Joh, Hel,He1,HKN,KuTa,HeNi,HeNi2,LeP,Dig,BHM,LeNi} 
for works in this spirit and the links between statistical mechanics
and Witten Laplacians. 
In particular, as was recognized in \cite{HKN}, the operator $\Delta^{(1)}_{f,h}$
acting on $1$-forms, 
being related for $h\ra 0$ to the energetic bottlenecks responsible for the slowdown 
of the underlying 
stochastic process, appears rather naturally when analyzing the low-lying eigenvalues
of $\Delta_{f,h}^{(0)}$.  
\\[0.1cm]
We emphasize that semiclassical techniques
as WKB expansions, Agmon estimates
and harmonic approximation
for Schr\"{o}dinger operators, used e.g. in~\cite{HKN}, 
  are generally not uniformly controlled in the limit $N\ra\infty$
  (see however \cite{MaMo}
for previous work on $\Delta^{(p)}_{f,h}$ in large dimension
under convexity assumptions on $f$ and \cite{SjoLarge1,SjoLarge2,HelDEA}).
Also for the specific model we consider here, the arguments of
\cite{HKN} do not carry over with uniform bounds in $N$.

\

\noindent
{\bf Comments on the techniques used in this paper}
\\[0.1cm]
\noindent
Though inspired by the supersymmetric approach of~\cite{HKN}, in this paper we do not make explicit use of $\Delta^{(1)}_{f,h}$. Indeed, a careful analysis of the energy $V$
permits to construct a very efficient global quasimode passing through the bottleneck and connecting the two minima
of $V$ (see Definition~\ref{defpsi}). This construction, together with a precise analysis of Laplace integrals in large dimension, enables us to give the upper bound of Theorem~\ref{th.rough}.
\\[0.1cm]
For the lower bound in Theorem~\ref{th.rough}, we depart from the semiclassical approach and
rather exploit perturbation techniques for fixed $h$. These permit,
even though for general $\mu>1$ the function $V$ is not convex outside a compact set, to reduce to the case of
a convex energy and then to apply the well-known Bakry-\'{E}mery criterion. 
We use here that the interaction part in the energy $V$ 
is strong enough to ensure good relaxation properties for large $N$.
Thus, roughly speaking, we regard the energy coming from the single particle double-well potential as a perturbation of the interaction part. This is opposed to
the perturbative regime considered in 
previous works as \cite{BoHe, BoHe2}: in these references the interaction constant $\mu$
is tuned in a way that it is rather the interaction part to become a perturbation 
of the single particle potential. 

\

\noindent 
The relevant quantity naturally appearing in the estimates leading to Theorem~\ref{th.main2} is 
the quadratic form 
$$
\mc E (\varphi)\ :=\ 
\frac{ \int_{\mathbb R^N}     |L_h \varphi|^2  \  d m_{h,N}   }{hN\int_{\mathbb R^N} |\nabla\varphi|^2  \ d m_{h,N}}\,.
$$
To connect to the existing literature, we point out that, after integration by parts, this quantity can be 
equivalently rewritten in the two forms
\begin{equation}\label{energy1}
\mc E (\varphi)\ =\   
\frac{\int_{\mathbb R^N}\Gamma^2(\varphi) \  d m_{h,N} }{\int_{\mathbb R^N} \Gamma(\varphi) \  d m_{h,N} }
\ =\
\frac{\int_{\mathbb R^N}\big(L_{h}^{(1)}\nabla \varphi\big)\cdot\nabla\vp\  \  d m_{h,N} }{hN\int_{\mathbb R^N} |\nabla\varphi|^2 \  d m_{h,N} }
\,.
\end{equation}
Here $\Gamma, \ \Gamma^2$ are respectively the carr\'e du champ operator and its iteration (see for example \cite{BaGeLe} for more details
about this notion) and  $L_{h}^{(1)}:= L_h \otimes{\rm Id}+
hN \,\Hess V$ is an operator acting on vector fields (i.e. $1$-forms),
related to $\Delta_{f,h}^{(1)}$ via ground state transformation. 
The last expression in~\eqref{energy1} can be generalized
by allowing, instead of $\nabla\vp$, more general, non-gradient vector fields. 
This is one of the main advantages of the supersymmetric approach
and is crucially exploited in works as~\cite{HKN,HeNi, LeP,LeNi}, or~\cite{Dig} in a discrete setting. 
In the arguments we give here we do not use this additional freedom
since we can work with the gradient of the quasimode already exploited in the proof of Theorem~\ref{th.rough} and thus streamline both the results and the presentation.

\

\noindent
For the proof of Theorem~\ref{th.main3} 
we shall  adopt the Schr\"{o}dinger point of view 
and thus work with $\Delta^{(0)}_{f,h}$. We combine here
standard localization techniques for the analysis of semiclassical 
Schr\"{o}dinger operators \cite{CFKS} and a two-scale analysis
naturally adapted to the structure of the energy $V$.

\

\noindent{\bf Plan of the paper}
\\[0.1cm]
The rest of the paper is organized as follows. In Section~\ref{section.2}
we discuss basic properties of the model 
and the precise relation between the diffusion operator $L_h$ and 
the Schr\"{o}dinger operator $\Delta^{(0)}_{f,h}$. Sections~\ref{section.3},~\ref{se.sharpspectralgap}
and~\ref{se.lowerbound} are respectively devoted to the
proofs of Theorems~\ref{th.rough},~\ref{th.main2} and~\ref{th.main3}.
\\[0.1cm]
Subsection~\ref{subsectionZ},
which might also  be of independent interest, 
provides a sharp Laplace-type asymptotics for 
the normalization constant $Z_{h,N}:= \int_{\mb R^N}    e^{-\frac{V(x)}{hN}}   dx   $ when $h\ra 0$ with 
uniform  control in $N$.

\section{Basic properties of the model}\label{section.2}

\subsection{Properties of $V$ and related Gaussian estimates}

The aim of this section is to fix our notation and to provide 
some basic background information on our model which we shall use throughout in the rest of the analysis.  
 
\

\noindent
We denote by $\lp\cdot, \cdot\rp$ the standard scalar product in $\mb R^N$, by $\|\cdot\|$,
or $|\cdot|$ when no ambiguity is possible,
the corresponding Hilbert norm and, more generally, for $p\in \mb N$ we write
\[    \|x\|_{p}   \   :=   \   \big(  \sum_{k=1}^N  |x_k|^p   \  \big)^{\frac 1p}    \   .      \]
The gradient, Hessian and (negative) Laplacian acting on functions in $\mb R^N$
are denoted respectively by $\nabla, \hess$ and $\Delta$.

\

\noindent
Throughout the paper we fix a $\mu>1$. The energy function $V$, defined in~\eqref{defV}, 
can be rewritten in a more compact notation as
\begin{equation}\label{compactV}
 V(x)     \   =    \      \frac 14\|x\|_4^4   \  +
     \      \  \frac 12 \lp  x, ( K  - 1)  x \rp     \   +    \    \frac N4   \   ,    \end{equation}
where $K:\mb R^N \ra \mb R^N$ is a normalised discrete Laplacian, defined by setting for
 $x\in \mb R^{N}$ and $k\in\{1, \dots, N\}$,
  \begin{equation}\label{defdiscreteLaplace}   (K  x)_k \  := \  \  \frac{\mu}{4\sin^2(\frac \pi N)} \big( \ 2 x_k \ - \   x_{k+1} \ - \  x_{k-1}    
 \  \big)   \    .  
 \end{equation}
It is understood that $x_{N+1}:=x_1$ and $x_{0}:=x_N$, which corresponds to periodic
boundary conditions. It holds $\lp x, Kx\rp =   \lp Kx, x\rp$ and,
according to our choice of sign, $\lp x, K x\rp \geq 0$. The operator $K$ is diagonalised
through the discrete Fourier transform 
$\hat x\in \mb R^N$ of $x\in \mb R^N$, defined by
\[    \hat x_k    \  :  =    \    \frac{1}{\sqrt N}   \sum_{j=1}^{N}    x_j \    e^{-\im 2\pi \frac{j}{N}k}     \ .     \]
More precisely we have for every $k\in \{0, \dots, N-1\}$, 
\begin{eqnarray}\label{eigenvaluesK}   (\widehat{Kx})_k   & =  &    \nu_k \  \hat x_k    \   ,   \    \     \   \text{ where }    \   
\nu_{k}    \   :=   \    \mu   \,  \frac{\sin^2(k\frac \pi N)}{\sin^2(\frac \pi N)}      \  . 
\end{eqnarray}
Note that $\nu_0=0$ is a simple eigenvalue of $K$ and that its smallest non-zero eigenvalue equals 
 $\mu$ for every 
$N\in \mb N, N\geq 2$. We shall denote by $P:\mb R^N\ra \mb R^N$
the projection onto the eigenspace of $K$ corresponding to the eigenvalue $0$
and by $P^{\perp} := 1 - P$ the projection onto its orthogonal complement. 
Note that $\range P={ \rm Span}(1,\dots, 1)$ so that $P$  
associates to $x\in \mb R^N$ the constant 
vector with components  the mean of $x$:  for every $ k\in \{1, \dots, N\}$,
\begin{equation} \label{projector}    (Px)_k    \  =    \  \overline x   \     :=    \      \frac 1N \sum_{j=1}^N x_j      \  =   \    
 \frac{\hat x_0}{\sqrt N }   \   .   
\end{equation}
For shortness the range $\range P$ of $P$ will sometimes also be denoted by $\mc C$, and we refer to it 
as the space of constant states, or the ``diagonal'' of $\mb R^N$. 
Similarly we write $\mc C^{\perp}:=\range P^{\perp}$ for the space
of states orthogonal to the constants. 

\

\noindent
We mention here explicitly the following simple identities, 
which we shall frequently use in the sequel:
\[ \forall x\in \mb R^N     \  ,  \   \     \sum_{k=1}^N (P^{\perp}x)_k    \  =   \   0      \   \   \   \text{ and }   \   \   \    \|Px\| \   =    \   \sqrt N 
\overline x    \ =  \    \hat x_0      \ .   \]
The fact that the first non-zero eigenvalue of $K$ equals $\mu$ 
implies the following  discrete Poincar\'{e}-type inequality: 
\begin{equation}\label{discretePoincare}    \forall \rho\in [0, \mu] \  , \  \forall x\in \mb R^N     \   ,    \   \   \    
\lp x,  Kx  \rp        \geq     \     \rho  \  \big( \   \|x\|^2    \   -  \    \lp x, Px \rp  \   \big) \  .   
\end{equation}

\noindent
For more information on the discrete Fourier transform
and discrete Laplacian, 
see for example \cite{Te}.

\

\noindent
Some basic features of the energy landscape determined by $V$
are the following. First, it is straightforward to check that the constant states given by 
\[     I_+   \  :=   \   (1,\dots, 1)    \  \  ,    \   \    I_-    \  :=  \
 (-1,\dots, -1)      \  \  \text{ and }     \   \    O    \  :=  \
 (0,\dots, 0)     \ ,    \]
 are critical points of $V$, i.e. satisfy $\nabla V(x)=0$, for every $N\in \mb N$. Moreover
\begin{equation}\label{hessians} \hess V (I_{\pm})   \ =   \      K+2  \   \      \text{ and  }      \      \   
 \hess V (O)   \ =   \      K-1             \  .     
 \end{equation}
It follows from~\eqref{eigenvaluesK} that $K+2$ admits only  strictly positive eigenvalues, while $K-1$ has one simple eigenvalue $-1$ and, since $\mu>1$,  all the others are strictly positive. The identities~\eqref{hessians} imply therefore in particular that $I_{\pm}$ are local minima and $O$ is a saddle point, i.e. 
a critical point of index $1$.  The additive constant $\frac N4$ appearing 
in~\eqref{compactV} is chosen such that
\begin{equation}\label{energyvalues}
V(I_\pm )    \ =  \  0    \   \     \text{ and }     \   \     V(O)  \   =   \    \frac N4     \  . 
\end{equation}

\noindent
A crucial feature of the model, implied by  the assumption $\mu>1$, 
is the following. When 
restricted to the $N-1$ dimensional subspace $\mc C^\perp = \range P^{\perp}$, 
 the quadratic form $\hess V$ is strictly convex, uniformly in $N$ and $x\in \mb R^N$.
Indeed, according to the discrete Poincar\'{e} inequality given in~\eqref{discretePoincare}, for every $x$ and $w\in \mb R^N$, 
\[     \lp w, \hess V(x) w\rp       \  \geq     \   \lp w , (K-1) w \rp  \  
 \geq     \     (\mu -1)    \   \|w\|^2    \   -  \    \mu \lp w, P w \rp  \  .      \]
In particular one gets the lower bound
 \begin{equation}\label{uniformconvexity}    \forall x\in \mb R^N \ ,  \  \forall w\in \mc C^{\perp}  \  ,   \   \  \   \lp w, \hess V(x) w\rp       \   \geq     \     (\mu -1)    \   \|w\|^2      \  .      
 \end{equation}

\noindent
The latter inequality can be used to show that 
 $I_{+}, I_-$ and $O$ are the only critical 
points of $V$ (see also \cite{BeFeGe1}):
\begin{lemma}
Fix $N\in \mb N$ and let $x\in \mb R^N\setminus\{O, I_{+}, I_-\}$. Then
$\nabla V(x)   \neq 0$. 
 \end{lemma}
\begin{proof}
Estimate~\eqref{uniformconvexity} implies that
for each $c\in \mb R$ there can be at most one critical point of
the restriction $V|_{ H_{c}}$ of $V$  to the 
hyperplane
$H_c := \{x\in \mb R^N: \overline x = c\}$. Since for every $c\in \mb R$ the constant vector
$\mbf c:= (c, \dots, c)\in \mc C$ satisfies
\[     \forall w\in \mc C^{\perp}      \  ,   \   \  \     \lp \nabla V(\mbf c) , w \rp  
\    =  \    (c^3-c) \sum_{k=1}^N w_k    \  =   \       0   \    ,   \   \text{ i.e }     \nabla \, (V|_{H_c} )(\mbf c) \ = \   0    \ ,     \]
the critical points of $V$ necessarily have to be on 
the diagonal $\mc C$. The statement of the lemma follows now by noting that for  $\mbf c:= (c, \dots, c)$
\[    V(\mbf c)   \  =   \     \frac 14 c^4    \    -   \     \frac 12 c^2     \  +   \    \frac 14    \    .\] 
\end{proof}
\noindent
Since, according to~\eqref{hessians}, the quadratic part of $V$ around its critical points is essentially given by the discrete Laplacian $K$, 
part of our analysis will rely on a good control in large dimension of 
Gaussian integrals, whose covariances are
given by the resolvent of $K$ or slight perturbations thereof.  
To be specific, we shall consider
 for suitable $\alpha, \beta\in \mb R$ operators $Q:\mb R^N \ra \mb R^N$ 
 of the form 
\begin{equation}\label{covariance}   Q  \    :  =   \  \big(  \    \alpha P     \   +    \    K    \   +    \beta   \   \big)^{-1}     \   ,     
\end{equation}
where $P$ is the projection  given by~\eqref{projector}.  
Note that the particular case $\alpha=0, \beta = 2$ corresponds to $Q=(K+2)^{-1}$, which 
according to~\eqref{hessians} equals the inverse of the Hessian of $V$ at the minima. Taking instead $\alpha=2$, $\beta=-1$, one obtains 
$Q=(2P + K-1)^{-1}$, which
is the inverse of the Hessian of $V$ at the saddle point, modulo inverting sign of its unique negative eigenvalue. 

\noindent
In general, for any choice of $\alpha, \beta$ such that 
 $Q$ is well-defined, 
it follows from~\eqref{eigenvaluesK} that for each $k\in\{0, \dots, N-1\}$
it holds $(\widehat{Qx})_k    \  =  \    \sigma_k \  \hat x_k$, where the eigenvalues are now given by 
\[   \sigma_{k}    \   =   \   \frac{1}{\nu_k+\beta}     \   \text{ for } 
k\in\{1, \dots, N-1\}    \   \  \text{ and }      \    \       \sigma_0 =\frac{1}{\alpha + \beta}  \  .    \]
In particular, $Q$ is positive in the sense of quadratic forms if and only if
$\alpha+\beta>0$ and $\mu+\beta>0$, which is assumed from now on. A crucial 
property of $Q$ is that it is of trace class, uniformly in the dimension:
\begin{equation}\label{traceclass}
\exists C> 0   :   \     \forall N\in \mb N  \   ,   \   \    \     \tr Q    \   :=   \   
\sum_{k=0}^{N-1}   \sigma_k     \     <     \  C       \       .          
\end{equation}
The latter estimate is obtained by straightforward  estimates on the 
$\nu_k $'s (essentially $\nu_k \sim \mu k^2$, see their expression in~\eqref{eigenvaluesK}). 
We remark en passant that~\eqref{traceclass}
fails to hold in the 
case of higher-dimensional single particle state, i.e. $x_k \in \mb R^d$ with $d>1$.
This is linked to well-known difficulties in the analytical treatment of the Stochastic 
Allen-Cahn equation in higher spatial dimension. 
A straightforward consequence of~\eqref{traceclass}
is a uniform control in $N$ of moments  of the 
centered Gaussian distribution with covariance $Q$
whose density is given by
$$
d\mu\,=\,\frac{e^{-\frac{\langle x, Q^{-1}x \rangle}{2}}}{\big((2\pi)^N\det Q\big)^\frac12}\,dx\,.
$$
In particular we will repeatedly exploit in this paper the following
uniform bound, which we state here for later reference.

\begin{lemma}\label{momentsgaussian}
Let $Q$ be defined as in~\eqref{covariance} with $\alpha+\beta>0$ and $\mu+\beta>0$. Then there exists a constant $C>0$ such that for $p\in\{4,6\}$ 
and every $h>0$, $N\in \mb N$,
 \begin{equation}\label{momentstatement}     \frac{1}{\big( (2\pi hN)^N \det Q
 \big)^{\frac 12}} \   \int_{\mb R^N}   N^{-1}\|x\|_p^p  \     e^{- \frac{\lp x, (hNQ)^{-1}x\rp}{2}}   \  dx       \   \leq \   C  h^{\frac p2}   \  .    
 \end{equation}
\end{lemma}

\

\begin{proof}
Differentiating suitably the moment generating function
of a Gaussian with covariance $Q$, given by
$$
\forall\,\xi\,\in\mathbb R^N\ ,\quad \frac{1}{\big( (2\pi)^N \det Q
 \big)^{\frac 12}} \   \int_{\mb R^N}   e^{- \frac{\lp x, Q^{-1}x\rp}{2}} e^{\lp x,\xi\rp}  \  dx 
 \ =\
 e^{\frac{\lp \xi, Q\xi\rp}{2}}      \     ,    
$$
yields  
for the left hand side of~\eqref{momentstatement} the expression
      \begin{equation}\label{momentgenerating}  \   \frac {C_p}{N}  \sum_{k=1}^{N} 
 (hNQ_{k,k})^{\frac p2}     \     ,    \   \  
 \text{ with   }   \  \  \  C_4=3,\  C_6 =15     \  .     \end{equation}
Using the ``Fourier integral representation" of $Q$  we get for its diagonal terms the expression
$Q_{k,k}    \  =     \    \frac 1N   \sum_{j=0}^{N-1}\sigma_j$ for every $k$.
It follows that~\eqref{momentgenerating} equals $  h^{\frac p2} C_p (\sum_{k=0}^{N-1} \sigma_k)^{\frac p2}$. This yields the desired result thanks to~\eqref{traceclass}. 
\end{proof}
\noindent
For convenience of the reader, we also state explicitly  
the following 
simple tail estimate, which will be exploited throughout the paper.
Recall that $\overline x$ denotes the mean of $x\in \mb R^N$ 
as defined in~\eqref{projector}.

\begin{lemma}\label{easytail}
 Let $Q$ be defined as in~\eqref{covariance} with $\alpha+\beta>0$ and $\mu+\beta>0$. Then for every $r>0$, the following estimate holds
 for $C=C(r)=\frac{(\alpha+\beta)r^2}{2}$ and for every $h\in (0,1]$
 and $N\in \mb N$,
\begin{equation*}     \frac{1}{\big( (2\pi hN)^{N}\det Q\big)^{\frac 12}}  \int_{\{|\overline x|>r\}}   e^{-\frac{\lp x, (hNQ)^{-1} x \rp}{2}}   \  dx   \  \leq     \    
\left(\frac{ h} {\pi C}\right)^\frac12
   e^{-\frac{C}{h}}            \  .   
\end{equation*}
\end{lemma}

\

\begin{proof}
Diagonalising $Q$ via the Fourier transform and recalling that $\overline x =  \frac{\hat x_0}{\sqrt N}$, we obtain
$$
 \frac{1}{\big( (2\pi hN)^{N}\det Q\big)^{\frac 12}}  \int_{\{|\overline x|>r\}}   e^{-\frac{\lp x, (hNQ)^{-1} x \rp}{2}}   \  dx 
 \ =\ 
\frac{2}{\sqrt{2\pi}} \int_{\{y_{0}>{r\sqrt{\frac{\alpha+\beta}{h}}\}}}   e^{-\frac{y_{0}^2}{2}}   \  dy_{0}
$$
and
the statement boils down to the standard 
Gaussian tail-estimate: \begin{equation}
\label{eq.tail}
 \forall\eta>0  \   ,   \   \  \  
\int_{\eta}^{+\infty}e^{-\frac{t^2}{2}}\,dt
\quad \leq\quad \frac{1}{\eta}e^{-\frac{\eta^2}{2}} \,.
\end{equation}

\end{proof}

\noindent
Lastly,
the ratio of the determinants of $\hess V(I_+)$ and $\hess V(0)$
converges, as already observed in \cite{St} (see also \cite{BaBoMe,BeGe}). More precisely the following statement holds true:
 \begin{lemma}
  \label{lemmaconvergenceratiodet}
 The relation \eqref{prefactor} mentioned in the introduction holds true: 
\begin{equation}\label{convergenceratiodet}
\sqrt\frac{\det \hess V(I_+)}{|\det \hess V(0)|}  \   =  \    
\sqrt\frac{\det \hess V(I_-)}{|\det \hess V(0)|}    \    \underset{N\to+\infty}{\longrightarrow}\ 
    \frac{ \sinh (\pi \sqrt{2\mu^{-1}})}{ \sin (\pi \sqrt{\mu^{-1}})}         \  .  
\end{equation}
\end{lemma}
\begin{proof}
According to~\eqref{hessians} and to~\eqref{eigenvaluesK}, we have
for $2\leq N\in \mb N$,
\begin{gather*}
  \frac{\det \hess V(I_\pm)}{|\det \hess V(0)|} 
  \  =    \     \frac{\det (K+2)}{|\det (K-1)|} 
  \  =   \   
\prod_{k=0}^{N-1}    \frac{\nu_k + 2}{|\nu_k-1|}   \    =   \      2  \, 
\prod_{k=1}^{N-1}  \frac{\nu_{k}+2}{\nu_k-1}   \,,
\end{gather*}
so we want to show that
$$
\sqrt {2} \sqrt{ 
\prod_{k=1}^{N-1}  \frac{\nu_{k}+2}{\nu_k-1}}=
\sqrt {2}
\prod_{k=1}^{N-1}\left(1+  \frac{3}{\nu_k-1}\right)^{\frac12}
\ \underset{N\to+\infty}{\longrightarrow}\ c_{\mu}\,,
$$
where $c_{\mu}$ is given by 
$$
c_{\mu}\ :=\ 
\frac{ \sinh (\pi \sqrt{2\mu^{-1})}}{ \sin (\pi \sqrt{\mu^{-1})}}
\ =\ 
\sqrt 2\,\prod_{k=1}^{+\infty}\frac{\mu k^2 +2}{\mu k^2 -1}
\,,
$$
the last equality being a direct consequence of 
Euler's product formula 
$$
\forall\,z\,\in\mathbb C\ ,\quad \sin(\pi z)\,=\,\pi z\prod_{k=1}^{+\infty}\left(1-\frac{z^2}{k^2}\right)\,. 
$$
Noticing now the relation
$$
\prod_{k=1}^{N-1}\left(
1+\frac{3}{\nu_{k} -1}\right)^\frac12
=
\left(
1+\frac{3}{\mu \frac{1}{\sin^2(\frac{\pi}{N})} -1}\right)^\frac{\textbf{1}_{2\mathbb N}(N)}{2}
\prod_{k=1}^{\lfloor{\frac{N-1}{2}}\rfloor}\left(
1+\frac{3}{\nu_{k} -1}\right)\,,
$$
we are then lead to prove that
$$
\prod_{k=1}^{\lfloor{\frac{N-1}{2}}\rfloor}\left(
1+\frac{3}{\nu_{k} -1}\right)
\ \underset{N\to+\infty}{\longrightarrow}\ 
\prod_{k=1}^{+\infty}\frac{\mu k^2 +2}{\mu k^2 -1}
\ =\ \lim\limits_{N\to+\infty}\prod_{k=1}^{\lfloor{\frac{N-1}{2}}\rfloor}\left(1+\frac{3}{\mu k^2 -1}\right)\,,
$$
and it is therefore sufficient to show that
\begin{equation}
\label{eq.sufficient}
\prod_{k=1}^{\lfloor{\frac{N-1}{2}}\rfloor}\frac{
1+\frac{3}{\nu_{k}-1}}
{1+\frac{3}{\mu k^2 -1}}
=
\prod_{k=1}^{\lfloor{\frac{N-1}{2}}\rfloor}\!\!\!\left(
1+3\frac{\nu_{k}-\mu k^2}{(\nu_{k}-1)(\mu k^2+2)}
\right)\ \underset{N\to+\infty}{\longrightarrow}\
1\,.
\end{equation}
The end of the proof follows
 from
the computations done in \cite{BaBoMe}~pp.~331-332 but we give the details
for the sake of completeness.
From the inequalities
$$
\forall\,x\,\in[0,\frac\pi2]\ ,\quad
0\,\leq\,x^2(1-\frac{x^2}{3})\,=\,x^2-\frac{x^4}{3}\,\leq\,\sin^2x\,\leq\, x^2\,,
$$
we deduce that for every $2\leq N\in\mathbb N$ and $k\in\{1,\dots,\lfloor{\frac{N-1}{2}}\rfloor\}$,
\begin{equation}
\label{eq.akn}
\mu\, k^2(1-\frac{\pi^2}{12})\,\leq\,\mu\, k^2\left(1-\frac{\pi^2k^2}{3N^2}\right)\,\leq\,\nu_{k}\,\leq\,\frac{3\mu N^2k^2}{3N^2-\pi^2}\,,
\end{equation}
and therefore
\begin{equation*}
-\frac{\mu\pi^2k^4}{3N^2}\,\leq\,\nu_{k}-\mu\,k^2\,\leq\,\frac{ \mu\pi^2k^2}{3N^2-\pi^2}\,,
\end{equation*}
from which we obtain that for every
$N\geq 2$ and
 $k\in\{1,\dots,\lfloor{\frac{N-1}{2}}\rfloor\}$,
 \begin{equation}
\label{eq.akn-k2}
|\nu_{k}-\mu\,k^2|\,\leq\,\frac{2\mu\pi^2k^4}{N^2}\,.
\end{equation}
It follows from \eqref{eq.akn} and \eqref{eq.akn-k2} that
there exist $2\leq N_{0}\in \mathbb N$
and a positive constant $C$ such that for every $N\geq N_{0}$
and $k\in\{\lfloor{\frac{N_{0}-1}{2}}\rfloor,\dots,\lfloor{\frac{N-1}{2}}\rfloor\}$,
$$
\left|3\frac{\nu_{k}-\mu\,k^2}{(\nu_{k}-1)(\mu k^2+2)}\right|\,\leq\, \frac{C}{N^2}\,.
$$
Using the inequality
$
|\ln(1+x)|\,\leq\,\frac{|x|}{1-|x|} 
$ valid on $(-1,1)$,
we get that for every $\mathbb N\ni N\geq \max\{N_{0},\sqrt C+1\}$,
$$
\Big|\ln \!\!\!\prod_{k=\lfloor{\frac{N_{0}-1}{2}}\rfloor}^{\lfloor{\frac{N-1}{2}}\rfloor}\!\!\!\left(
1+3\frac{\nu_{k}-\mu k^2}{(\nu_{k}-1)(\mu k^2+2)}
\right)\!\!\Big|\,\leq
\sum_{k=\lfloor{\frac{N_{0}-1}{2}}\rfloor}^{\lfloor{\frac{N-1}{2}}\rfloor}\frac{C}{N^2-C}\,\underset{N\to+\infty}{\longrightarrow}\,
0\,,
$$
and the equation \eqref{eq.sufficient} we were looking for follows,
since for any fixed $k$, $1+3\frac{\nu_{k}-\mu k^2}{(\nu_{k}-1)(\mu k^2+2)}$ goes to $1$ when $N\to+\infty$.
\end{proof}
\noindent
For additional background on Gaussian measures and perturbations thereof in large and infinite dimensions we point to~\cite{GlJa,Sim,DP}.

\subsection{Relation between $L_{h}$ and the Witten Laplacian}
\label{subse.Witten}

\noindent
As already mentioned in the introduction,
all the results stated there
can equivalently be reformulated in terms of Witten Laplacians
using the ground state transformation.
 More precisely, up to a multiplicative factor $hN^{-1}$,
the operator
 $L_h:=-   h N \Delta      +      \nabla V \cdot \nabla$ 
 acting on $L^2(e^{-\frac{V}{hN}}dx)$
 is unitarily equivalent to a semiclassical Witten Laplacian  acting on the flat
$L^2(dx)$:
\begin{equation}
\label{eq.linkWittenflattilde}
 e^{-\frac{ \tilde f}{h}}\  h\, L_{h}\    e^{\frac{ \tilde f}{h}}     \    =\    N\big(-   h^2 \Delta     \     +    \ 
 |\nabla \tilde f|^2      \  -    \       h    \Delta \tilde f\big)
 \ =:\  N\,\Delta^{(0)}_{\tilde f,h}
        \   , 
\end{equation}
where
\[    \tilde f (x)    \    : =   \   \   \frac {V( x  )}{2N}        \  .     \]
Using in addition 
the unitary
dilatations ${\rm Dil}_{\lambda}$ on $L^2(dx)$, which are
defined, for any $\lambda>0$ and any $g\in L^2(dx)$,
by ${\rm Dil}_{\lambda} g:=\lambda^{\frac N2} g(\lambda\,\cdot)$,
we have also the unitary equivalence
\begin{equation}
\label{eq.linkWittenflat}
{\rm Dil}_{\sqrt N}\ N\Delta^{(0)}_{\tilde f,h}\ {\rm Dil}_{\frac{1}{\sqrt N}}\ 
=\ -   h^2 \Delta        \  +\      
 |\nabla f|^2      \  -   \         h    \Delta  f
 \ =:\  \Delta^{(0)}_{f,h}\,,
\end{equation}
where
\[    f (x)    \    : = \   \tilde f(\sqrt N x)  \ =   \      \frac {V(\sqrt N x  )}{2N}        \  .     \]
Note that  $\Delta^{(0)}_{f,h} = \sum_j (\partial_j + h\partial_j f)^*\, 
(\partial_j + h\partial_j f)$ with domain $C_c^\infty(\mb R^N)$ is symmetric and nonnegative in $L^2(dx)$, that its 
Schr\"{o}dinger potential $|\nabla f|^2 -h\Delta f$ is smooth and, 
 for fixed $h$ and $N$,  tending to infinity as $|x|\ra \infty$ .
It follows then from standard arguments of the theory of Schr\"{o}dinger operators
(see e.g. \cite[Proposition 7.10 and Theorem 9.15]{He2}) that $\Delta^{(0)}_{f,h}$ is essentially 
selfadjoint on $C_c^\infty(\mb R^N)$, and that its closure, which we still denote by
$\Delta^{(0)}_{f,h}$,
has compact resolvent and therefore purely discrete spectrum. 
Notice moreover that $\ker \Delta_{f,h}^{(0)}\, =\, {\rm Span}(e^{-\frac fh})\,$.

\

\noindent
Due to~\eqref{eq.linkWittenflattilde} and~\eqref{eq.linkWittenflat}, 
these properties can be immediately transferred   from $\Delta^{(0)}_{f,h}$
to $L_h$ and it holds
\begin{equation}
\label{eq.spectrumLhWitten}
\spec(\Delta^{(0)}_{f,h})\ =\ h\spec(L_{h})\,.
\end{equation}

\section{Uniform bounds in the dimension}\label{section.3}

This section is devoted to the establishment of  lower bounds on the
log-Sobolev constant $\rho_{h,N}$ (defined through~\eqref{logsobdef}) and upper bounds on the
spectral gap $\lambda_{h,N}$ (defined through~\eqref{specgapdef}), which are uniform in the system 
size~$N$. The main results here are the following.
\begin{theorem}[Lower Bound]\label{asymptoticlowerbound}

\noindent
For every $\delta>0$, 
there exists a positive constant $C_{\delta}$  such that the log-Sobolev constant
$\rho(h,N)$ satisfies
 \[     \forall      h>0 \   ,    \forall   N\in \mb N  \      ,     \   \      \  C_{\delta}  
   \ e^{-\frac{3+2\sqrt 2+\delta}{24h}}   \   
   e^{-\frac{1}{4h} }      \    \leq    \   \rho(h,N)     \  .        \]
\end{theorem}

\begin{theorem}[Upper Bound]\label{asymptoticupperbound}

\noindent
The spectral gap $\lambda(h,N)$ satisfies for every $h>0$ and $N\in \mb N$ the inequality
\[     \lambda(h, N)    \   \leq      \  p (N)    \ e^{-\frac {1}{4h}} 
\    \big(   \   1   \  +   \   \epsilon(h,N)   \  \big)      \       ,        \]
where the prefactor $p(N)$ is given by \eqref{prefactor} and   the error term $\epsilon(h,N) $ satisfies 
\[    \exists\, C>0   \  \text{ s.t. }  \  \forall h\in (0,1]  \   ,   \    \forall N\in \mb N     \   ,     \   \   \    \       |\epsilon(h,N)|     \   \leq    \      C  \   h          \  .   \]
\end{theorem}

\noindent
Note that, together with the well-known inequality $\rho(h,N)\leq \lambda(h,N)$,
which is easily obtained
by applying~\eqref{logsobdef} to $1+\ve u$ and letting $\ve\ra 0$,
Theorem~\ref{asymptoticlowerbound} and Theorem~\ref{asymptoticupperbound}
yield Theorem~\ref{th.rough}. 

\

\noindent
The proof
of Theorem~\ref{asymptoticlowerbound} is based on 
a careful perturbation of the energy $V$ and a combination of the
Holley-Stroock perturbation principle and the Bakry-\'{E}mery criterion for log-Sobolev constants
(c.f. Proposition~\ref{HSperprin} and Proposition~\ref{BakryEmery} below).
The proof of Theorem~\ref{asymptoticupperbound}  relies
on a suitable choice of a test function (or ``quasimode'') and exploits a good control on the normalisation constant   $Z_{h,N}:= \int_{\mb R^N}e^{-\frac{V(x)}{hN}} dx$.

\

\noindent
This section is organised as follows. Subsection~\ref{subsectionLB} contains  the proof
of Theorem~\ref{asymptoticlowerbound}. Subsection~\ref{subsectionZ}, 
which might also  be of independent interest, 
provides a sharp Laplace-type asymptotics for $Z_{h,N}$ when $h\ra 0$ with 
uniform  control in $N$. Finally, Subsection ~\ref{subsectionUB} contains the proof of Theorem~\ref{asymptoticupperbound}.

\subsection{Proof
of Theorem~\ref{asymptoticlowerbound} (Lower Bound on $\rho_{h,N}$)}\label{subsectionLB}

Our proof is based on a combination
of the following two well-known criteria for establishing lower bounds on the 
log-Sobolev constant
(see for example \cite[Prop. 3.1.18 and Theorem 3.1.29]{Ro}
or the original papers  \cite{HoSt},  \cite{BaEm}).  

\

\noindent
We fix here $N\in \mb N$ and use the following standard notation: for a measurable function $U:\mathbb R^N\ra\mathbb R$
such that $e^{-U}\in L^1(\mb R^N, dx)$,
 we define the probability measure $dm_U:= Z^{-1}_U e^{-U} dx$, 
where   $   Z_U  \   :=    \     \int_{\mb R^N} e^{-U}   \  dx   $ is the normalisation constant. 
Moreover we write for nonnegative $u\in \mc C_{{\rm c}}^\infty(\mb R^N;\mb R)$,
\[         \ent_{m_U} (u)    \   :=   \    \int_{\mb R^N}    u \log u   \  dm_U       \  -   \    \int_{\mb R^N} u\log \left( \int_{\mb R^N}  
u \ dm_{U}  \right) \ dm_{U}    \  ,      \]
and define $\rho_{U}$ as the largest positive constant
such that  $\forall u\in \mc C_{{\rm c}}^\infty(\mb R^N;\mb R)$,
\begin{equation}\label{standdeflogsob} 
           \rho_{U}   \    \ent(u^2)      \      \leq   \   2
 \       \int_{\mb R^N}  |\nabla u|^2   \      dm_{U}     \     .        
\end{equation}

\begin{proposition}[Holley-Stroock perturbation principle]\label{HSperprin}

\

\noindent
Let $U:\mathbb R^N\ra\mathbb R$
s.t. $e^{-U}\in L^1(\mb R^N, dx)$ and let $W:\mb R^N\ra \mb R$ be a bounded measurable function. Then 
\[      \rho_{U}      \   \geq     \       e^{- (\sup W - \inf W )}   \  \rho_{U+W}       \   . \] 
\end{proposition}

\begin{proposition}[Bakry-\'{E}mery criterion]\label{BakryEmery}
Let $U \in \mc C^2(\mb R^N)$ such that (in the sense of quadratic forms)
\[    \exists\, C>0 \ : \ \forall x\in \mb R^N    \   ,   \   \   \        \hess U(x)     \   \geq   \   C     \  .    \] 
Then the log-Sobolev constant satisfies
\[     \rho_U    \     \geq      \   C  .     \]
\end{proposition}

\noindent
In order to prove Theorem~\ref{asymptoticlowerbound}, we construct 
a suitable perturbation $W$ 
which added to our energy $V$ produces a strictly convex function. 
This is done as follows.   
First, for each $n\in \mb N$, we take  some even $\theta_n\in \mc C^2(\mb R;[0,1])$
satisfying 
\begin{equation}
\label{cutoffproperties0}
\theta_n(r)    \    =   \    \begin{cases}    1    \  \   \    \text{ if } |r| \leq 1 \\
0     \   \   \   \text{ if } |r|   \geq  \sqrt 2 
\end{cases}\!\!, \quad    \quad    \quad \theta_n'(r)\leq 0 \quad \text{if}\quad r\geq0\,, 
\end{equation}
and, for every $r\in\mathbb R$,
\begin{equation}\label{cutoffproperties}
\theta_{n}''(r) \geq - \frac{2}{(\sqrt 2-1)^2}  (1 +\frac1n)     \    .
\end{equation} 
This is indeed possible since, by elementary arguments, one can check that
 $$
\sup\!\left\{
\min_{r\in\mathbb R} f''(r)\, ;\, f \in \mc C^2(\mb R;[0,1])\  
\text{is even and satisfies \eqref{cutoffproperties0}}\!\right\}=-\frac{2}{(\sqrt 2-1)^2}\,.$$
Next, in order to ``convexify'' $V$ with  some as small as possible perturbation, 
we consider for every $n\in \mb N$, $\alpha\in (0,1) $ and $\beta>0 $ 
the family of perturbations 
\begin{gather}\label{familypert23}    \   W_{\alpha,\beta,n}(x)   \  :  =     \   
 \sum_k    \theta_n( c_{\alpha, \beta} x_k)   \    \Big(        
    \
   -  \   \frac{1-\alpha}{4}x_k^4   \  +  \  \frac{1+\beta}{2}  x_k^2    \    \Big)    \   -  \    \frac{N}{4}   \  ,  
     \end{gather}
 where   $c_{\alpha, \beta} := \sqrt{\frac{1-\alpha}{1+\beta}} $.  
 Since the polynomial part of  
 \eqref{familypert23} is nonnegative on $\supp \theta_n( c_{\alpha, \beta} \,\cdot) $, one gets from $0\leq\theta_{n}\leq1$
 the  
 two  bounds  
 \begin{equation}\label{boundedperturbation}    
-\frac N4    \   \leq   \     W_{\alpha, \beta, n}(x)    \    \leq    \   \frac N4  \ \Big( \ \frac{(1+\beta)^2}{1-\alpha} \  -  \ 1  \ \Big)     \  ,     \         
\end{equation}
valid for every  $n\in \mb N, \alpha\in (0,1), \beta>0$   and every $ x\in \mb R^N $.   Moreover, for a suitable choice of the parameters $\alpha, \beta$ and $n$, the  $W_{\alpha, \beta, n}$--perturbation
of the original energy $V$ becomes a uniformly strictly convex function:
 
  \begin{lemma}\label{lemmaconvexperturbation2}
Let $\alpha\in (\frac{1}{3(2-\sqrt2)^2+1}, 1)$. Then there exists $n_0\in \mb N$ such that for any $\beta>0$ and $n\geq n_0$, we have in the sense of quadratic forms:
\[          \exists C_{\alpha,\beta, n} >0 \  \text{s.t.}   \         \  \forall x\in \mb R^N\,,\ \forall N\in \mb N   \   ,  \    \    \        \hess \big(   V + W_{\alpha, \beta, n}  \big)   (x)   \   \geq     \   C_{\alpha,\beta,n }     \       .   \]
\end{lemma}

\begin{proof}Recalling the definition \eqref{defV}  of   $V$ and that  
 the discrete Laplacian $\mu K$ is nonnegative we get the estimate
\[    \forall x \in\mathbb R^N  \ ,   \   \forall \alpha, \beta>0  ,\  \forall n\in\mb{N}      \ ,    \   \    \     \hess \big(   V + W_{\alpha, \beta, n}  \big)   (x)    \    \geq       \    \hess   U_{\alpha, \beta, n}(x)     \   ,    \]
where
\[     U_{\alpha, \beta, n}     \   :=    \      \frac 14 \sum_k 
\big( \  1-   (1-\alpha)\theta_{n}(
c_{\alpha, \beta} x_k)\  \big)  \    x_k^4   \   -    \   \frac 12 \sum_k \big( \  1-(1+\beta)\theta_{n}(
c_{\alpha, \beta} x_k) \ 
\big)   \    x^2_k     \   . 
\]
The Hessian of $U_{\alpha, \beta,n}$ is diagonal and we have,
for any $k\in\{1,\dots,N\}$: 
\begin{align*}     \partial^2_k U_{\alpha, \beta, n}    \   (x)      \  =   \    &    
 \underbrace{ c_{\alpha, \beta}^2  \theta_n''( c_{\alpha, \beta} x_k)   \    \Big(        
      \  
     -  \   \frac{1-\alpha}{4}x_k^4   \  +  \  \frac{1+\beta}{2}  x_k^2    \    \Big)}_{\text{I}}   \   +   \\
      &   \   +    \
     \underbrace{ 2c_{\alpha, \beta} \theta_n'(c_{\alpha, \beta} x_k)   \   \    \Big(        
    -  \  (1-\alpha) x_k^3   \  +  \  (1+\beta)  x_k    \    \Big)}_{\text{II}}      \   +  \\
    &  \ 
    +    \ 
          \underbrace{ \theta_n(c_{\alpha, \beta} x_k)     \    \Big(        
    -  \   3 (1-\alpha)x_k^2   \  +  \  (1+\beta)    \    \Big)    \   +   \
     3 x_k^2   \  -  \  1 }_{\text{III}}  \   .
    \end{align*}

\noindent
{\it Case 1}:  $|c_{\alpha, \beta} x_k| >  \sqrt 2$. 

\noindent
Then $\theta_n=\theta_n'=\theta_n''=0$ for every $n\in \mb N$ and  we obtain
\[ \forall \alpha, \beta>0    \  ,     \    \    \   \partial_k^2    U_{\alpha, \beta, n}(x)    \   =    \   3x_k^2   \   -    \   1   \    \geq    \   \frac{6(1+\beta)}{1-\alpha}
\  -   1   \    \geq   \   5   \  .   \]

\

\noindent
{\it Case 2}:  $|c_{\alpha, \beta, n} x_k| <  1$. 

\noindent
Then $\theta_n'=\theta_n''=0$, $\theta_n = 1$  for every $n\in \mb N$ and  we obtain
\[\forall \alpha, \beta>0    \  ,     \    \    \ \partial_k^2    U_{\alpha, \beta, n}(x)    \   =   \ 
 \    3\alpha x_k^2   \   +   \beta     \    \geq    \   \beta      \  .   \]

\

\noindent
{\it Case 3}:  $c_{\alpha, \beta} x_k \in [1,\sqrt2]$.  

\noindent
First, for every $\beta>0, \alpha\in (0,1)$, $n\in \mb N$, we have, using $\theta_n'(c_{\alpha, \beta}x_k)\leq 0 $ (see indeed~\eqref{cutoffproperties0}) and
$\big(    -    (1-\alpha) x_k^3   +    (1+\beta)  x_k      \big) \leq 0$,  that $\text{II} \geq 0$.\\
Moreover, we deduce from $\big(   3 (1-\alpha)x_k^2   -    (1+\beta)\big)  \geq 0$ that
 the term $\text{III}$ satisfies
\[     \text{III}    \   =   \  
    ( 1 -  \theta_n(c_{\alpha, \beta} x_k))     \    \Big(        
      \   3 (1-\alpha)x_k^2   \  -  \  (1+\beta)    \    \Big)    \   +   \
     3 \alpha x_k^2 +\beta   \ 
     \geq\   3 \alpha x_k^2 +\beta  \   .   \]

\

\noindent
Let us lastly look at the term $\text{I}$. 
Since 
$\big(   -    \frac{1-\alpha}{4}x_k^4    +   \frac{1+\beta}{2}  x_k^2    \big)\geq0$, we have 
\begin{gather*} 0\ \leq\  c_{\alpha, \beta} ^2   \    \Big(        
      \  
     -  \   \frac{1-\alpha}{4}x_k^4   \  +  \  \frac{1+\beta}{2}  x_k^2    \    \Big)    \   =   \  
      x_k^2    \   c_{\alpha, \beta} ^2  \    \Big(        
      \  
     -  \   \frac{1-\alpha}{4}  x_{k}^2   \  +  \  \frac{1+\beta}{2}     \    \Big)     \\
     \leq   \    x_k^2    \      \Big(        
      \  
     -  \   \frac{1-\alpha}{4}     \  +  \   \frac{1-\alpha}{2}     \    \Big)   \    =     \   
        \   \frac{1}{4}  (1-\alpha)   \  x_k^2
            \  ,  
     \end{gather*}
and so, since $\theta_n''(c_{\alpha, \beta}  x_k)\geq \frac{-2}{(\sqrt 2-1)^2}(1 +   \frac1n)$
   according to~\eqref{cutoffproperties},
\[   \text{ I}    \       \geq    \   -  (1 +\frac1n) \frac{1-\alpha}{(2-\sqrt2)^2}    \  x_k^2  
   \,. \]
Summing up,  we then have in {\it Case 3}:
\begin{gather}
\label{estimatealphabetan}
\forall \alpha, \beta>0  \  ,     \    \    \ \partial_k^2    U_{\alpha, \beta, n}(x)    
  \geq     \   
     \frac{\alpha\big(3(2-\sqrt2)^2+1 +\frac1n \big) -
 1 -\frac1n}{(2-\sqrt2)^2}   \  x_{k}^2+\beta \,.
    \end{gather}
If $\alpha>\frac{1}{3(2-\sqrt2)^2+1}$ as in the assumption, there exists $n_0\in \mb N$ such that
the right hand side of~\eqref{estimatealphabetan} is bigger than $\beta$
and hence strictly positive for any $n\geq n_{0}$.

\

\noindent 
The case  $c_{\alpha, \beta} x_k \in [-2,-1]$ 
can be treaten in an analogous way and thus the lemma is proven. 
\end{proof}

 \noindent
 The proof of Theorem~\ref{asymptoticlowerbound} can now be easily concluded: for any $\delta>0$, taking some fixed
   $\alpha\in (\frac{1}{3(2-\sqrt2)^2+1},1)$, 
   we have for $\beta>0$ sufficiently small,
  \begin{equation} \label{deltaalphabeta}  \frac{1+\delta}{4} +\frac{3+2\sqrt 2}{24} \ =\ 
  \frac{3(2-\sqrt2)^2(1+\delta)+1}{12(2-\sqrt2)^2}
   \   \geq     \      \frac{(1+\beta)^2}{4(1-\alpha)}     \  .    \end{equation}
 According to Lemma~\ref{lemmaconvexperturbation2}, fixing
$n$ sufficiently large, there exists a $C_{\delta}>0$ such that
the perturbation $W_{\alpha, \beta, n}$ defined in~\eqref{familypert23}
satisfies uniformly with respect to 
the dimension~$N\in\mathbb N$ and
to $x\in\mathbb R^N$:
\[     \hess (V + W_{\alpha, \beta, n})  (x)  \   \geq  \   C_{\delta}  \  .   \]
Moreover, by estimates~\eqref{boundedperturbation} and~\eqref{deltaalphabeta}, 
\begin{equation*}
     \sup_{x}\frac{W_{\alpha, \beta, n}(x)}{Nh}    \   -   \  
     \inf_{x}  \frac{ W_{\alpha, \beta, n}(x)}{Nh}    \  \leq    \ 
   \frac{1+\delta}{4} +\frac{3+2\sqrt 2}{24}   \     .    \end{equation*}
Applying the perturbation principle as stated in Proposition~\ref{HSperprin}
with $U=\frac{V}{hN}$ and $W= (W_{\alpha, \beta, n})/hN $ yields therefore
 \[     \rho_{V/hN}    \    \geq     \    \frac{C}{hN} \   e^{-\frac{1+\delta}{4h} -\frac{3+2\sqrt 2}{24h}}        \        .   \]
Noting that the rescaled log-Sobolev constant $\rho(h,N)$ as defined in~\eqref{logsobdef}
satisfies 
\[       \forall       h>0 \ ,  \   \forall  N\in \mb N  \      ,     \   \      \       \rho(h,N)   \     =   \    hN
 \rho_{V/hN}      \       ,   \]
we get the statement of Theorem~\ref{asymptoticlowerbound}.

\subsection{Computation of the normalisation constant $Z_{h,N}$}\label{subsectionZ}

\noindent
To obtain a good quantitative upper bound,
we are lead to compute precise Laplace asymptotics. 
Similar computations are done in \cite{BaBoMe} exploiting the Hausdorff-Young inequality. We follow a different route 
based
on a comparison with a suitable quadratic form (see \eqref{quadraticbelow} below)
and
giving better error estimates.\\

\noindent
Note first the
expressions for $V$ shifted to the minima,
\noindent
\begin{equation} \label{shiftedV} V(x  + I_{\pm})     \   =    \      \frac 14\|x\|_4^4   \     \pm
\   \sum_k x_k^3     \  +     \  \frac 12 \lp  x,(K +2) x \rp     \   ,   \end{equation}

\noindent
and let $Q$ be the  following operator that will be used to control $V$ from below in the rest of this subsection:
\begin{equation}\label{defQ}  
 Q  \    :  =   \  \big(  \    \frac 32 P     \   +    \    K    \    -1   \   \big)^{-1}         \   .  
 \end{equation}
 
\noindent
This linear operator then satisfies in particular
 \begin{equation}\label{detofQ}
\det Q^{-1}   \  =   \   \frac 12 |\det (K -1)| \    =    \    \frac 12  |\det \hess V(0)|      \   .   
\end{equation}

\begin{lemma} \label{lemmaquadraticbelow}
Let $Q:\mb R^N\ra \mb R^N$ be the linear operator
defined through equation~\eqref{defQ}. Then the following two estimates hold:
\begin{equation}\label{quadraticbelow}
\forall \, \overline x \geq -1    \    ,     \     \     V(x+I_+)      \   \geq     \   
  \frac 12   \lp  x, Q^{-1} x\rp           \       ,  
\end{equation}
and
\begin{equation}\label{quadraticbelowabs}  
\forall\,  |\overline x| \leq 1    \    ,      \      \    \frac 14 \|x\|_4^4   \   -    \   \big|\sum_k x_k^3\big|   \  +  \   \frac12\lp x,( K + 2)x \rp       \   \geq      \   
 \frac 12   \lp  x, Q^{-1}x\rp  
      \     .
\end{equation}
\end{lemma}
\begin{proof}
Note first the following estimate implied by H\"older's inequality:\begin{equation}\label{basicinequality}
\forall x\in \mb R^N    \  ,   \   \  \   \|x\|_4^4     \   \geq   \         N  \overline x^4     \   . 
\end{equation}
It follows that 
\[    \frac 14 \|x\|_4^4   \   -  \   \frac N2 \overline x^2     \  +   \   \frac N4  \   \geq   \  
  \frac N4 \overline x^4   \   -  \   \frac N2 \overline x^2   \  +   \frac N4  \
  =   \     \frac N4 (\overline x -1)^2(\overline x+1)^2       \ ,    \]
  and therefore
  \begin{eqnarray}    V(x)    \  &\geq&   \    \frac N4 (\overline x -1)^2(\overline x+1)^2     \   +
  \ \frac 12 \lp  x,  ( K-1) x \rp 
     \   +  \   \frac N2 \overline x^2   \       \nonumber  \\
     &=& \   
       \frac N4 (\overline x -1)^2(\overline x+1)^2     \   +
  \ \frac 12 \lp  x ,  (P + K-1)  x \rp    
      \    ,        \label{potentialestimate}
     \end{eqnarray}
where the last inequality follows from the relation $N\overline x^2  =    \lp x, Px
\rp$.\\    
From~\eqref{potentialestimate}, since $P + K-1$ annihilates constants,  
we get for the shifted potential  the estimate   
\begin{equation}
V(x+I_+)      \   \geq     \    \frac N4 \overline x^2   \  +  \    
 \frac 12 \lp  x,  (P +  K-1) \ x \rp      \   \  \  \text{ for }     \overline x \geq -1    \   ,  
\end{equation}
which proves~\eqref{quadraticbelow}. 
Note moreover that~\eqref{potentialestimate} also gives
\begin{equation}\label{quadraticbelow2}
\forall \, \overline x \leq 1    \    ,     \     \     V(x+I_-)      \   \geq     \   
  \frac 12   \lp  x, Q^{-1} x\rp           \       ,  
\end{equation}
which is actually equivalent to~\eqref{quadraticbelow} due to the symmetry of $V$.
Estimate~\eqref{quadraticbelowabs} is then an immediate consequence
of the expressions for $V(x+I_+)$ and $V(x+I_-)$ given in~\eqref{shiftedV}
and~\eqref{quadraticbelow},~\eqref{quadraticbelow2}. 
\end{proof}

\begin{proposition}\label{tailpropN}
For every $r>0$ there exists a constant $C>0$ such that 
for each $h\in (0, 1]$ and $N\in \mb N$,
\begin{equation}\label{taillargeN}     \int_{\{\overline x\geq 0\,;\,|\overline x - 1|\geq r \}}   e^{-\frac{V(x)}{hN}}   \  dx    \  
\leq    \       C^{-1}       \       \frac{(2\pi hN)^{\frac N2}}{|\det \hess V(I_+)|^{\frac 12}}      \   e^{-\frac Ch}      \ .   
\end{equation}
\end{proposition}

\

\begin{proof} Fix $r>0$. Shifting the origin to the minimum $I_+$ and using the quadratic lower bound given in~\eqref{quadraticbelow} of Lemma~\ref{lemmaquadraticbelow}, we get
\[     I      \  : =    \  \int_{\{\overline x\geq 0\,;\,|\overline x - 1|\geq r \}}   e^{-\frac{V(x)}{hN}}   \  dx      \   \leq    \  
\int_{\{\overline x\geq -1\,;\,|\overline x |\geq r \}}   e^{-\frac{\lp x, (hNQ)^{-1}x\rp}{2}}   \  dx      \   ,   \]
where $Q:\mb R^N\ra \mb R^N$ is the positive operator defined in~\eqref{defQ}. 
According to the Gaussian tail estimate of Lemma~\ref{easytail}, 
there exists a constant $C>0$ such that for every $h\in (0,1]$
and $N\in \mb N$, 
\begin{equation*}    I   \  \leq     \     C^{-1}   \frac{ \  (2\pi hN)^{\frac N2}}{\big(\det Q^{-1}\big)^{\frac 12}}      e^{-\frac Ch}            \  .   
\end{equation*}
The desired result follows now from \eqref{detofQ} and 
the convergence of the ratio of determinants
given by~\eqref{convergenceratiodet} of  Lemma~\ref{lemmaconvergenceratiodet}. 
\end{proof}

\begin{proposition}\label{laplaceprop}
Let $r\in(0,1]$. Then  
\begin{equation}\label{laplace}    
\int_{\{|\overline x-1|\leq r\}}   e^{-\frac{V(x)}{hN}}   \  dx    \     =       \       \frac{(2\pi hN)^{\frac N2}}{|\det \hess V(I_+)|^{\frac 12}}     \    \big(   \     1
+   \    \epsilon_r(h,N)     \   \big)    \         ,   
\end{equation}
where the error term $\epsilon_r(h,N) $ satisfies 
\begin{equation*}   
   \exists\, C=C(r)>0 \  \text{ s.t. } \  \forall h\in (0,1]  \   ,   \    \forall N\in \mb N     \   ,     \   \   \    \       |\epsilon_r(h,N)|     \   \leq    \      C \   h          \  .   
\end{equation*}

\end{proposition}

 \begin{proof}
Fix $r\in (0,1]$. Shifting the origin  to the minimum $I_+$ (recall~\eqref{shiftedV}) and isolating the contribution
of the integral given by the non-quadratic part of $V$ from the rest, we write
\begin{gather}    
\int_{\{|\overline x-1|\leq r\}}   e^{-\frac{V(x)}{hN}}   \  dx    \    =      \      
\int_{\{|\overline x|\leq r\}}   e^{-\frac{V(x+I_+)}{hN}}   \  dx  \qquad\qquad
\qquad\qquad
\qquad\qquad
\qquad\qquad
\qquad\qquad
  \nonumber \\
\qquad\qquad=   
 \underbrace{\int_{\{|\overline x|\leq r\}}   
   e^{-\frac{\lp x ,  ( K+2) x  \rp}{2hN}} 
     \  dx}_{=:I}      \  +  \    
      \underbrace{\int_{\{|\overline x|\leq r\}}   a(x)     \    
      e^{-\frac{\lp x ,  ( K+2) x  \rp}{2hN}} 
     \  dx  }_{=:II}    \    ,    \label{integraldecomposition} 
\end{gather}
with
\[     a(x)    \  :=   \       \exp{\big( \ -\frac{   \frac 14\|x\|_4^4  \  +  \   \sum_k x_k^3  }{hN} \  \big)  }     \ -     1   \  .    \]

\

\

\noindent
{\it Computation of $I$:}

\noindent
For the integral $I$ appearing in~\eqref{integraldecomposition}, 
recalling that $\hess V(I_+)=K+2$ and using the Gaussian tail estimate of Lemma \ref{easytail}, we obtain the existence of $C>0$, such that
for every $h\in (0,1]$, $N\in \mb N$, 
\begin{equation}\label{computationI}    I   \  =    \    \frac{ \  (2\pi hN)^{\frac N2}}{|\det \hess V(I_+)|^{\frac 12}}     \    \big(   \     1
+   \    \epsilon_r(h,N)     \   \big)          \  ,    
\end{equation}
where the error term $\epsilon_r(h,N)$ satisfies
\begin{equation*}  
   \exists\, C=C(r)>0 \  \text{ s.t. } \  \forall h\in (0,1]  \   ,   \    \forall N\in \mb N     \   ,     \   \   \    \       |\epsilon_r(h,N)|     \   \leq    \      C^{-1} \   e^{-\frac Ch}          \  .   
\end{equation*}

\

\noindent
{\it Estimate of $II$:}

\noindent 
For the integral $II$ appearing in~\eqref{integraldecomposition},  we proceed as follows: evaluating
the estimate
\[   \forall  \,   t\in \mb R  \  ,     \   \   \    |e^t  -1  -  t|   \  \leq    \     \frac 12 t^2 e^{|t|}   \       \]  
at $t= -(hN)^{-1}\sum_kx_k^3$, we get
\begin{equation} \label{decomptaylor}    a(x) \  =  \  
 \underbrace{e^{-  \frac{  \|x\|_4^4  }{4hN} }  -1}_{=:A}    \   -  \   \underbrace{(hN)^{-1} \sum_k x_k^3  \   e^{ -  \frac{  \|x\|_4^4}{4hN} }}_{=:B}  \   +   \    \epsilon(h,N,x)  \    ,   
 \end{equation}
with
\begin{equation}\label{errorestimate8}|\epsilon(h,N, x) |   \     \leq    \   \frac{1}{2 (hN)^2}  \|x\|_3^6 \ 
\exp{\big( -  \frac{  \frac 14 \|x\|_4^4 -  \big|   \sum_k  x_k^3 \big|}{hN}\big) }       \   .     
\end{equation}
Using  $\|x\|_3^6\leq N \|x\|_6^6$ and~\eqref{quadraticbelowabs} in Lemma~\ref{lemmaquadraticbelow},
it follows from~\eqref{errorestimate8} that
\begin{equation}\label{errorestimate9}
e^{-\frac{\lp x ,  ( K+2) x  \rp}{2hN}}
|\epsilon(h,N, x) |   \     \leq    \   \frac{1}{2 h^2N}  \|x\|_6^6    \    
e^{-\frac{  \lp x, (hNQ)^{-1}x\rp}{2}}        \   ,     \end{equation}
where $Q:\mb R^N \ra \mb R^N$ is defined in~\eqref{defQ}. 
Using  for the term $A$ appearing in \eqref{decomptaylor}
the inequality  $0  \leq 1 - e^{-|t|} \leq |t|$, 
the antisymmetry of the term $B$ and estimate~\eqref{errorestimate9} to control
 $\epsilon(h,N, x)$, 
we get
 \begin{equation*}    |II|   \  \leq   \   \frac{1}{4hN} \ \int_{\mb R^N}
 \|x\|_4^4  \   \    e^{-\frac{\lp x ,  (K+2) x  \rp}{2hN}} 
     \  dx  \  +   \    \frac{1}{2h^2 N}   \       \int_{\mb R^N}       \|x\|_6^6   \ 
  e^{-\frac{\lp x, (hNQ)^{-1}x\rp}{2}  }   \  dx
  \    .      \end{equation*} 
The statement of  Lemma~\ref{momentsgaussian} about 
the control of moments of Gaussian integrals 
together with the expression of $\det Q^{-1}$ given in~\eqref{detofQ} yield
the existence of a $C>0$ such that for every $h>0$, $N\in \mb N$, 
\[   |II|      \   \leq      \  C     \    h  \   \frac{(2\pi h)^{\frac N2}}{
|\det \hess V (I_+)|^{\frac 12} }    \   +   \  
 C  \   h  \   \frac{(2\pi h)^{\frac N2}}{
|\det \hess V (0)|^{\frac 12} }       \  .    \]   
Recalling the convergence of the ratio of determinants
given by~\eqref{convergenceratiodet} of  Lemma~\ref{convergenceratiodet}, we finally obtain  
\begin{equation}\label{estimateII}  
\exists\, C>0 \  :    \   \    \forall\, h>0\,,\ \forall\, N\in \mb N   \  ,     \    \   \     |II|   \  \leq     \  C  \  h  \   \frac{(2\pi h)^{\frac N2}}{
|\det \hess V (I_+)|^{\frac 12} }    \  .    
\end{equation}  
Putting together~\eqref{computationI}   and~\eqref{estimateII} gives the statement of the proposition.

\end{proof}

\noindent According to the symmetry of $V$,
Propositions~\ref{tailpropN} and~\ref{laplaceprop}
finally lead to the precise computation of the normalisation 
constant $Z_{h,N}$:

\begin{corollary}
For the normalisation constant $Z_{h,N}$ we have
\begin{equation}\label{laplace}    
Z_{h,N}   \  :=  \   \int_{\mb R^N}   e^{-\frac{V(x)}{hN}}   \  dx    \     =       \      2 \frac{(2\pi hN)^{\frac N2}}{|\det \hess V(I_+)|^{\frac 12}}     \    \big(   \     1
+   \    \epsilon(h,N)     \   \big)    \       \    ,   
\end{equation}
where the error term $\epsilon(h,N) $ satisfies 
\[   \exists\, C>0 \  \text{ s.t. } \  \forall h\in (0,1]  \   ,   \    \forall N\in \mb N     \   ,     \   \   \    \       |\epsilon(h,N)|     \   \leq    \      C  \   h          \  .   \]
\end{corollary}

\subsection{Upper Bound on $\lambda_{h,N}$}\label{subsectionUB}

We give in this section the proof of Theorem~\ref{asymptoticupperbound}. We recall that
for $x\in \mb R^N$, 
\[    \overline x   \  :=   \     \frac 1N\sum_{k=1}^N x_k      \  ,  \]

\noindent
and we consider in the rest of this subsection the following operator $Q$, 
whose inverse is $\Hess V(O)$ modulo inverting sign of its unique
negative eigenvalue,

\begin{equation}\label{changesign}   Q      \  :=    \       (2P + K -1)^{-1}       \   .       
\end{equation}
We have then in particular the relation
 \begin{equation}\label{detofQ2}
\det Q^{-1}   \  =   \    |\det (K -1)| \    =    \     |\det \hess V(0)|      \   .   
\end{equation}

\begin{definition}\label{defpsi}
Let $\chi= \chi_{h,N}: \mb R^N\ra [-1,1]$ 
be the function defined by
\[     \chi(x)    \  :=    \         \frac{2}{\sqrt{2\pi hN}} \int_0^{\sqrt N \overline x}  e^{- \frac{t^2}{2hN}}   \  dt  \  =  \  
 \frac{2}{\sqrt{2\pi h}} \int_0^{\overline x} e^{- \frac{t^2}{2h}}   \ dt    \   .   \]
For $h>0$ let $\psi= \psi_{h,N}:\mb R^{N} \ra \mb R$ be given by
  \begin{equation*}
   \psi (x)   \   :=   \   \frac{\chi(x)}{   \big(  \  \int_{\mb R^N} \chi^2( x) \  
   e^{-\frac{V(x)}{hN}}   \  dx \ 
   \big)^{\frac 12}}    \   .  
\end{equation*} 
\end{definition}

\begin{remark}\label{orthogonalquasimode}
Note that by antisymmetry, the quasimode $\psi$ 
has mean zero:
\[   \int_{\mb R^N}     \psi(x)  e^{-\frac{V(x)}{hN}}     \  dx   \  =    \       0     \  .   \]
\end{remark}

\begin{lemma}\label{compunorm}
The square of the weighted $L^2$-norm of $\chi$ satisfies 
\begin{equation*}\label{qsmest2}  
\int_{\mb R^N}   \chi^2( x) \  e^{-\frac{V(x)}{hN}}   \  dx    \      =       \      2\  \frac{(2\pi hN)^{\frac N2}}{|\det \hess V(I_+)|^{\frac 12}}     \    \big(   \     1
+   \    \epsilon(h,N)     \   \big)    \       \    ,   
\end{equation*}
where the error term $\epsilon(h,N) $ satisfies
\begin{equation}\label{errorOh}    \exists C>0   \  \text{ s.t. }  \  \forall h\in (0,1]  \   ,   \    \forall N\in \mb N     \   ,     \   \   \    \       |\epsilon(h,N)|     \   \leq    \      C  \   h          \  .   \end{equation}
\end{lemma}

\begin{proof}
By the symmetry of $V$ and $\chi^2$ and splitting the integral we get
\begin{gather*}  \int_{\mb R^N} \chi^2( x) \  e^{-\frac{V(x)}{hN}}   \  dx      \     =   \    
2 \ \int_{\{\overline x \geq 0\}} \chi^2( x) \  e^{-\frac{V(x)}{hN}}   \  dx  \qquad\qquad\qquad\qquad\qquad\qquad\\ 
\qquad
=  \   \underbrace{2   \    \int_{\{|\overline x -1| \leq \frac 12\}} \chi^2( x) \  
e^{-\frac{V(x)}{hN}}   \  dx }_{=:I}  \   +  \ 
\underbrace{2  \ \int_{\{\overline x\geq 0\,;\,|\overline x -1| \geq \frac 12\}} \chi^2( x) \  e^{-\frac{V(x)}{hN}}   \  dx   }_{=:II}   \,  .   
\end{gather*}
Using for the term $I$ the simple estimate
\[    \exists\, C>0 \   :     \   \forall x \in \{|\overline x-1 | \leq \frac 12   \}  \,,\  \forall \,
h\in (0,1]   \  ,  \  \   \  \     
|\chi(x)-1|   \    \leq    C^{-1} \  e^{-\frac Ch}     \  ,    \]
and  Proposition~\ref{laplaceprop}, and for the term $II$ the bound  $|\chi|\leq 1$ and
Proposition~\ref{tailpropN}, we get
\begin{equation}\label{qsmest2}  
\int_{\mb R^N}   \chi^2( x) \  e^{-\frac{V(x)}{hN}}   \  dx    \      =       \      2\  \frac{(2\pi hN)^{\frac N2}}{|\det \hess V(I_+)|^{\frac 12}}     \    \big(   \     1
+   \    \epsilon(h,N)     \   \big)    \    ,   
\end{equation}
where the error term $\epsilon(h,N) $ satisfies~\eqref{errorOh}.
\end{proof}

\noindent
Theorem~\ref{asymptoticupperbound} is then a direct consequence of
the following proposition:

\begin{proposition}\label{quasimodal1}
The function $\psi$ from Definition~\ref{defpsi} satisfies for every $h>0$ and every $N\in \mb N$,
\begin{equation*} 
hN   \   \int_{\mb R^N} |\nabla \psi|^2   \   e^{-\frac{V(x)}{hN}}   \   dx    \   =    \    \ \frac 1\pi   \  \left|\frac{\det\Hess V(I_-)}{\det\Hess V(0)}\right|^{\frac12} \   e^{-\frac{1}{4h}}   \   
\big(   \   1   \  +   \   \epsilon(h,N)   \  \big)       \  ,
\end{equation*}
where  the error term $\epsilon(h,N)$ satisfies 
\begin{equation*}\label{errorOhtt}    \exists\, C>0   \  \text{ s.t. }  \  \forall h\in (0,1]  \   ,   \    \forall N\in \mb N     \   ,     \   \   \    \       |\epsilon(h,N)|     \   \leq    \      C  \   h          \  .   \end{equation*}
\end{proposition}

\begin{proof}
Since for every $x\in \mb R^N$, 
\[       hN \, |\nabla \chi|^2 (x)    \  =   \      \frac 2\pi   \  
e^{-\frac{\overline x^2}{h}}  
\  =   \     \frac 2\pi   \  e^{-\frac{\lp x, 2P x\rp}{2hN}}      \  ,  \]
we get with $Q$ as defined in~\eqref{changesign}, 
\begin{gather*}
hN   \,   \int_{\mb R^N} |\nabla \chi|^2   \   e^{-\frac{V(x)}{hN}}   \   dx    \
=
  \   \int_{\mb R^N}    e^{-\frac{\lp x, (hNQ)^{-1}x\rp}{2}}     \   dx 
  \qquad\qquad\qquad\qquad\\
\qquad \qquad\qquad   +     \   \int_{\mb R^N}    e^{-\frac{\lp x, (hNQ)^{-1}x\rp}{2}}   \
\big(    \   e^{-\frac{\|x\|_4^4 }{4hN}}    -1  \   \big)             \   dx    \   .
\end{gather*}
From this equality, the expression of the determinant of $Q$
given in~\eqref{detofQ2}, and from
the inequality  $0  \leq 1 - e^{-|t|} \leq |t|$,   we obtain 
the following estimate also using
the uniform bounds on Gaussian
moments provided by Lemma~\ref{momentsgaussian},
\begin{equation}\label{qsmest1}
hN   \,   \int_{\mb R^N} |\nabla \chi|^2   \   e^{-\frac{V(x)}{hN}}   \   dx    \   =    \    \ \frac 2\pi   \ 
\frac{(2\pi hN)^{\frac N2}}{|\det\Hess V(0)|^{\frac12}}   \   \big(   \   1   \  +   \   \epsilon(h,N)   \  \big)       \  ,
\end{equation}
where the error term $\epsilon(h,N) $ satisfies~\eqref{errorOhtt}. Combining
this with Lemma~\ref{compunorm}
finishes the proof.
\end{proof}

\section{Sharp spectral gap asymptotics}
\label{se.sharpspectralgap}

In this section, we prove  Theorem~\ref{th.main2}. 
To do so, we will again
use the test function $\psi$ introduced in Definition~\ref{defpsi}
in order to show that it asymptotically saturates
the inequality
\[\forall\,\varphi \in D(L_{h})\quad \text{s.t.}\quad 
\|\varphi\|_{L^2(e^{-\frac{V}{hN}})}=1
 \ ,\quad  \lambda(h,N)    \  \leq   \  hN\,\int_{\mathbb R^N} |\nabla\varphi|^2 e^{-\frac{V(x)}{hN}}    dx     
 \]
under a further
assumption on
the separation between the second and the third eigenvalues of $L_{h}$.
Under this condition, we can indeed reverse 
this inequality up to an error term involving 
the quadratic form 
$$
\mc E (\varphi)\ :=\ 
\frac{ \int_{\mathbb R^N}     |L_h \varphi|^2  \  e^{-\frac{V(x)}{hN}}   \  dx }{hN\int_{\mathbb R^N} |\nabla\varphi|^2 e^{-\frac{V(x)}{hN}}\ dx}\,,
$$
which was already mentioned in the introduction.

\begin{proposition}\label{spectralgaprepresentation}
Let $\delta,h_0>0$ and, for every  $h\in (0, h_0]$,  $\mc N(h)\subset \mb  N$    s.t.
\begin{equation}\label{conditiondelta}   
  \forall h\in(0,h_0]  \   ,   \    \forall N\in \mc N(h)      \   ,     \   \   \    
 \spec(L_h)   \   \cap  \   [0,\delta)         \       =       \   \{0,\lambda(h,N)\}       \    .
 \end{equation}
Then, for all  $h\in(0,h_0]$, $N\in \mc N(h)$ and
$\varphi:= \varphi_{h,N}\in D(L_h)$ satisfying
\begin{equation}\label{proptest}\int_{\mathbb R^N}\!\! \varphi^2 e^{-\frac{V(x)}{hN}} dx  =    1     \,  ,
\ 
     \int_{\mathbb R^N}\!\! \varphi \,e^{-\frac{V(x)}{hN}}    dx  =    0\ ,
     \ 
     \int_{\mathbb R^N}\!\! \varphi (L_h \varphi)   e^{-\frac{V(x)}{hN}}   dx\ <\ 
     \frac{\delta}{2} \, ,  
\end{equation}
we have the lower bound
\[     \lambda(h,N)    \  \geq   \  hN\int_{\mathbb R^N} |\nabla\varphi|^2 e^{-\frac{V(x)}{hN}}dx   \   \big(  \   
1   \  -   \     \epsilon(h,N)   \  \big)     \  ,      \]
where the error term $\epsilon(h,N)$ satisfies
\[            0 \  \leq  \ \epsilon(h,N)     \   \leq    \  
\min\big\{\,1\ ,\ \frac{2hN}{ \delta} \int_{\mathbb R^N} |\nabla\varphi|^2 e^{-\frac{V(x)}{hN}}dx     +    \frac{2}{  \sqrt\delta}  \ \sqrt{\mc E(\varphi)} \,  \big\} \   .  \]
\end{proposition}

\noindent
The proof is a simple application of the following standard Markov-type inequality,
which is a consequence of the spectral theory for self-adjoint operators. 

\begin{lemma}\label{lemmafuncanalysis}
Let $T$ be a nonnegative self-adjoint 
 operator on a Hilbert space $H$ 
with domain $D$. Then for every $u\in D$ and every $b>0$,
\[\|1_{[b,\infty)}(T) u \|^2 \ \leq\  \frac{\lp T u, u \rp}{b} \ \ . \]
\end{lemma}
 
\

\begin{proof}[Proof of Proposition~\ref{spectralgaprepresentation}]
We denote respectively by $\lp\cdot, \cdot \rp$ and $\|\cdot\|$ the scalar product and the Hilbert norm
in $L^2(e^{-\frac{V}{hN}})$, and by
$P := \mbf 1_{[0,\delta)}(L_h)$ the spectral projector of $L_h$
onto the interval $[0,\delta)$. From
$$
\int_{\mathbb R^N} \varphi (L_h \varphi)   e^{-\frac{V(x)}{hN}} \   dx\ 
=\ hN\int_{\mathbb R^N} |\nabla\varphi|^2 e^{-\frac{V(x)}{hN}}\ dx \,,
$$
using the third point of   the 
property~\eqref{proptest} together with 
Lemma~\ref{lemmafuncanalysis}, 
we get
\begin{equation}\label{normprojectionnew}     
   \|(1-P)\varphi\|^2  \ =\ \|\mbf 1_{[\delta,+\infty)}(L_h)\varphi\|  
    \    \leq      \      \frac{hN}{\delta}\|\nabla\varphi\|^2   
     \   <     \     \frac 12      \   .      
\end{equation}
In particular, since $\|\varphi\|=1$, we have also
\begin{equation}\label{normprojection2}
\|P\varphi\|^2    =   \|\varphi\|^2  -   \|(1-P)\varphi\|^2   \geq \frac12  
\   ,   \end{equation}
and so
$P\varphi\neq 0$. We can therefore define
$ u  :=    \frac{P\varphi}{\|P\varphi\|}$.
Since moreover $\lp \varphi, 1\rp=0$, we have
$\lp P\varphi, 1\rp = \lp\varphi, P1\rp=0\,.$ 
Thus, using also~\eqref{conditiondelta}, 
$u$ is necessarily a normalised eigenfunction of $L_h$ associated with
the 
eigenvalue $\lambda(h,N)$. Consequently, it follows from $L_hP=PL_h$
on $D(L_{h})$, the self-adjointness
of $P=P^2$, and elementary rearrangements of terms, that
\begin{gather*}    \lambda(h,N)    \  =    \   \lp u , L_h u  \rp     \   =   \    \frac{\lp P\varphi, L_h P\varphi \rp}{\|P\varphi\|^2}    \   =
\  \frac{\lp \varphi, L_h \varphi\rp}{\|P\varphi\|^2}    \  +  \  \frac{\lp P\varphi -\varphi, L_h \varphi\rp }{\|P\varphi\|^2} 
 \\
= \  hN \| \nabla \varphi\|^2  \Big[    1  \  +   \  
\underbrace{\frac{\|(1-P)\varphi\|^2}{\|P\varphi\|^2}}_{=:I}    \  +  \  
\underbrace{\frac{\lp (P-1)\varphi, L_h \varphi \rp}{hN\| \nabla \varphi\|^2}}_{=:II}      
\underbrace{\big(   1  \  +   \  \frac{\|(1-P)\varphi\|^2}{\|P\varphi\|^2} \big)}_{=:III}  \Big] .
\end{gather*}
The statement of the proposition follows now by observing that, according 
to~\eqref{normprojectionnew} and~\eqref{normprojection2}, 
\[   I    \    \leq   \  \frac{2hN}{\delta}\|\nabla\varphi\|^2     \   ,       \   \       
\left| II\right|    \    \leq   \  \frac{ \|L_h\varphi\|}{\sqrt{\delta\,hN}\|\nabla\varphi\|}    \  ,   \   \text{and}\quad     III   \    \leq     \    2     \ ,     \]
and $\lambda(h,N)$ is nonnegative.
\end{proof}

\noindent
Applying Proposition~\ref{spectralgaprepresentation}
with the test function $\psi$ as  defined in  Definition~\ref{defpsi},
the statement of Theorem~\ref{th.main2} is then
a direct consequence 
of the quasimodal estimates
given in Proposition~\ref{quasimodal1} and 
in the following proposition:

 \begin{proposition}\label{Propquasimodal2}
 Let $\psi$ be the test function introduced in Definition~\ref{defpsi}.
Then there exists $C>0$ such that for every $h\in(0,1]$ and
every  $N\in \mb N$, 
\begin{equation}\label{quasimodal2}
 \  \int_{\mathbb R^N}     |L_h \psi|^2  \  e^{-\frac{V(x)}{hN}}   \  dx    \    \leq         \         C    \,      h^2    \,  
   \left|\frac{\det\hess V(I_-)}{\det\hess V(0)}\right|^{\frac12} \   e^{-\frac{1}{4h}}  \  .  
\end{equation}
\end{proposition}

\begin{proof}
A straightforward computation, whose
 details are given below for the sake of
completeness, leads to the identity 
 \begin{equation}\label{identityquasimodal2}    \int_{\mb R^N}     |L_h \chi|^2  \  e^{-\frac{V(x)}{hN}}   \  dx     \  =   \   \frac{2}{\pi h N^2}   \   
 \int_{\mb R^N}   \big(\sum_{k=1}^Nx_{k}^3\big)^2   \   e^{-\frac{V(x) + \lp x, P x \rp}{hN}} \  dx   \   .    
 \end{equation}
Hence, 
using the estimate
$\big(\sum_{k=1}^Nx_{k}^3\big)^2\leq   N \|x\|_6^6$
implied by the Cauchy-Schwarz inequality, 
we obtain the bound
\begin{equation} \label{boundsquasimodal2}
  0   \     \leq     \   \int_{\mb R^N}     |L_h \chi|^2  \  e^{-\frac{V(x)}{hN}}   \  dx     \    \leq   \      \   \frac{2\,e^{-\frac{1}{4h}}}{\pi h N}   \   
 \int_{\mb R^N}    \|x\|_6^6    \   e^{-\frac{\lp x, (hNQ)^{-1}x\rp}{2}} \  dx   \   ,    
 \end{equation}
 where $Q$ is defined in~\eqref{changesign}. The estimate~\eqref{quasimodal2}
 follows by applying the uniform moment bound of Lemma~\ref{momentsgaussian}  
 to the right hand side of~\eqref{boundsquasimodal2}, 
 recalling the expression~\eqref{detofQ2} of the determinant of $Q$ and 
 finally invoking Lemma~\ref{compunorm} for $\int_{\mathbb R^N} \chi^2 e^{-\frac{V(x)}{hN}}dx$. 
 
 \

\noindent
To show~\eqref{identityquasimodal2} and thus completing the proof, we 
note that for
$k\in\{1, \dots, N\}$, 
\[         \partial_k \chi(x)  \  =   \    \frac{1}{ N} \frac{2}{\sqrt{2\pi h}}      e^{ -  \frac{\overline x^2}{2h}}    \   =    \   \frac{1}{ N} \frac{2}{\sqrt{2\pi h}}      
e^{ - \frac{\lp x, Px \rp}{2hN}}  \  ,   \]
and compute
\begin{eqnarray*}
L_h \chi  &= &  
 -  hN    \  e^{\frac{V(x)}{hN}}  \   \sum_{k=1}^N \partial_k \big(  \  e^{-\frac{V(x)}{hN}}   \partial_k  \chi    \ \big)      \\   
&=& -      \frac{2}{\sqrt{2\pi h}N}      \  e^{-\frac{\lp x, Px\rp}{2hN}} \ 
\sum_{k=1}^N \big(   - \partial_k V(x)   \  +     \     \overline x   \big)    \\ 
&= &
    \frac{2}{\sqrt{2\pi h}N}      \  e^{-\frac{\lp x, Px\rp}{2hN}} \ 
\sum_{k=1}^N  x_k^3   \       \ 
 \   ,        
\end{eqnarray*}
where for the last inequality we used $\sum_{k=1}^N (Kx)_k=0$. Taking the square 
we get~\eqref{identityquasimodal2}. 
\end{proof}

\section{Lower bound on the second spectral gap} \label{se.lowerbound}

The aim of this section is to prove Theorem~\ref{th.main3}. Instead of working directly 
with the diffusion operator $L_h$, we switch to the Schr\"{o}dinger operator 
point of view and consider the semiclassical Witten Laplacian 
on functions acting in the flat $L^2(dx)$ and given by  
\[\Delta_{f,h}^{(0)}   \  =    \   -h^2\Delta    \  +   \     |\nabla  f |^2 \      -h\Delta  f     \     , \]
where $f:\mb R^N \ra \mb R$ is defined as 
\begin{equation}\label{definitionfWitten}  
  f(x)   \  :=   \     \frac{V(\sqrt Nx)}{2N}  \  =   \      \frac N8 \|x\|_4^4   \  +  \  
\frac 14 \lp x, (K-1)x\rp    \     + \     \frac 18   \  .  
\end{equation}
Note that due to the rescaling of variables, the two minima of $f$ are rescaled
by a factor $\sqrt N$ with respect to the minima of $V$. More precisely they are given 
by 
\[       J_{+}    \   :=      \    \frac{I_+}{\sqrt N}  \ =    \  \frac{1}{\sqrt N} (1, \dots, 1)  \  \   \ ,  
\   \  \     J_{-}    \   :=      \    \frac{I_-}{\sqrt N} 
 \ =    \  \frac{1}{\sqrt N} (-1, \dots,- 1)        \  .  \]
Since in this proof we deal only with the Witten Laplacian acting on functions 
we drop in the sequel 
the superscript $(0)$
and write for short  $\Delta_{f,h}:=   \Delta_{f,h}^{(0)}$. 
\noindent  
Moreover, note also that from the relation~\eqref{eq.spectrumLhWitten}
 between $L_h$ and $\Delta_{f,h}$, Theorem~\ref{th.main3} is implied  by the 
following. 

\begin{theorem}\label{th.CFKS} Let $C>0$ and $\alpha\in (0, \frac 34)$. Then there exist two positive constants
$h_0$ and $\ell$ such that 
\[\forall h\in (0, h_0] \  \text{ and }   \  \forall N \leq C \, h^{-\alpha} \ ,   \    \ \    
\dim \big( \range \mbf 1_{[0, \ell h)} (\Delta_{f,h}) \big)    \  \leq     \   2       \  .  \]
\end{theorem}

\noindent
According to the Max-Min principle (see for example \cite[Theorem 11.7]{He2}), in order   to prove Theorem~\ref{th.CFKS}, it is sufficient to show that
there exist 
 $h_0, \ell>0$ such that
 for every  $h\in (0, h_0]$ and $N  \leq    C h^{-\alpha}  $, 
 there exist
 $E_+, E_-\in L^2(\mb R^N)$
   s.t.
for any $\psi\in \mc C^\infty_{{\rm c}}(\mb R^N;\mb R)$, 
\begin{equation}\label{goal}     \lp \psi, \Delta_{f,h}\psi   \rp   \    \geq      \    \ell   \,  h     \ 
\Big(   \   \|\psi\|^2_{L^2(\mb R^N)}    \ -    \  \lp \psi, E_+   \rp^2_{L^2(\mb R^N)}    \  -     \  
\lp \psi, E_ -  \rp^2_{L^2(\mb R^N)}     \   \Big)      \     . 
\end{equation}
To obtain estimate~\eqref{goal} we first follow a standard ``decoupling'' approach
by introducing a suitable partition of unity allowing to
split the integral on the left hand side of~\eqref{goal} into integrals 
over almost disjoint sets. These will be    localized   respectively
around the two minima of $f$, around the diagonal $\mc C$ but far from the minima, 
and far from the diagonal. The main tool here is the so-called IMS localization formula (see \cite{CFKS}). 

\begin{proposition}{(IMS Localization Formula)}
\label{pr.IMS}~\\
Let $d\in \mb N$ and $\{\eta_k\}_{1,\ldots,d}$ be a quadratic partition of unity 
of $\mb R^N$, i.e. such that $\eta_k\in \mc C_{{\rm c}}^\infty(\mb R^N)$ for every $k$ and $\sum_{k=1}^d \eta^2_k\equiv 1$. Then for every $\psi\in \mc C_{{\rm c}}^\infty(\mb R^N)$, 
\begin{equation}
\label{eq.IMS}
\lp \psi , \Delta_{f,h} \psi   \rp_{L^2(\mb R^N)}
   \   =     \  
\sum_{k=1}^{d}  \lp  \,  \eta_{k} \psi,  \,  \Delta_{f,h}   (  \eta_{k} \psi)   \,    \rp_{L^2(\mb R^N)}
        -   
h^2  \  \| \,  |    \nabla\eta_{k}|   \,      \psi\|_{L^2(\mb R^N)}^2    \;.
\end{equation}
\end{proposition}
\noindent
The second main ingredient to obtain estimate~\eqref{goal} relies on the 
decomposition $\mb R^N   =     \mc C    \oplus \mc C^{\perp}$ and 
on a two-scale approach. 
We recall  that $\mc C= \range P$ is one-dimensional where
$P$ has been defined in~\eqref{projector}.
For any $\psi\in \mc C_{{\rm c}}^{\infty}(\mb R^N; \mb R)$, we then  have 
the decomposition
\begin{equation} \label{decompositiontwoscales}  \Delta_{f,h}   \psi    \   =   \    \Delta^{\mc C}_{f,h}  \psi     \ +  \   
\Delta^{\mc C^{\perp}}_{f,h }  \psi        \    ,   
\end{equation}
where
\[\Delta^{\mc C}_{f,h}  \  :=      -h^2\Delta^{\mc C}    \  +   \     |\nabla^{\mc C}  f |^2 \      -    \    h\Delta^{\mc C}  f  \   \   \text{ and }   \    \   \Delta_{f,h}^{\mc C^{\!\perp}}   \  :=    -h^2\Delta^{\mc C^{\!\perp}}      +   \     |\nabla^{\mc C^{\!\perp}} \!\! f |^2 \      -    \    h\Delta^{\mc C^{\!\perp}}     \!\!  f     . \]
Here, the superscripts $\mc C, \mc C^{\perp}$ on a differential
operator mean that differentiation is restricted to the corresponding subspace. 
Thus, chosing some normalized coordinate $y_0$ on $\mc C$
and orthonormal coordinates $(z_1, \dots, z_{N-1})$ on $\mc C^\perp$ we have
for example
for every $\psi\in \mc C^\infty(\mb R^N)$, 
\[      \Delta^{\mc C} \psi    \   =   \         \frac{\partial^2  \psi  }{\partial y_0^2}    \  \   
   \    ,    \   \  \     \Delta^{\mc C^{\perp}} \psi    \   =   \         \sum_{k=1}^{N-1} \frac{\partial^2 \psi  }{\partial z_k^2}        \   .     \]
Note in particular that the orthogonal decomposition
$$x\ =\ Px\,+\,P^\perp x\ =\  \hat x_{0}\left(\frac{1}{\sqrt N},\dots,\frac{1}{\sqrt N}\right)\,+\,P^\perp x
$$ 
leads to
\begin{equation}
\label{eq.nablax0}
\nabla^{\mc C}  f(x)\ =\  \frac 12 \sqrt N \sum_{k=1}^N \big(\frac{\hat x_{0}}{\sqrt N}+P^{\perp}x\big)_k^3\,-\,\frac12\hat x_{0}
\end{equation}
and
\begin{equation}
\label{eq.laplx0}
\Delta^{\mc C}  f(x)
\ =\  \frac12\big(3\hat x_{0}^2\,+\, 3\|P^\perp x\|^2  \,-\,1\big)\,.
\end{equation} 
Given $\psi:\mb R^N \ra \mb R$ and $y\in \mc C$, we denote by $\psi_y$ 
the partial application
\begin{equation}\label{conditionednotation} 
  \psi_y  \  :\ z\,\in\,\mc C^{\perp}\quad  \longmapsto\quad      \psi(y + z)      
\end{equation}
and hence satisfying  $\psi_y(\mc C^{\perp})=\psi(\{x\in\mb R^N\,:\, Px = y\})$.

\noindent
Roughly speaking, we shall exploit decomposition~\eqref{decompositiontwoscales} as follows.
Away from the diagonal, we use $\Delta_{f,h}^{\mc C}\geq 0$
and exploit~\eqref{uniformconvexity}, namely that $\hess^{\mc C^\perp}$ is strictly convex, uniformly in $x$
and $N\in \mb N$. This leads for every fixed $y\in \mc C$ to spectral gap
lower bounds for the operator $\psi_y\mapsto (\Delta^{\mc C^\perp}_{f,h} \psi)_y$
(see Lemma~\ref{lemmaconditionalpoincare} below). 
The  dependence on $y$ of these estimates is controlled in Lemma~\ref{le.concent}
below. Around the diagonal $\mc C$ but away from the critical points
we use $\Delta_{f,h}^{\mc C^\perp}\geq 0$
and work with $\Delta_{f,h}^{\mc C}$, which, when restricted to sufficiently small neighbourhoods of $\mc C$, behaves essentially like the $1$-dimensional 
Witten Laplacian associated with $f|_{\mc C}$ (see the discussion after~\eqref{arounddiagonal} below).
Around the minima $J_+$ and $J_-$ we work directly with $\Delta_{f,h}$. 
Here we use that the restriction of $f$ to sufficiently small neighbourhoods around $J_+$ and $J_-$ is uniformly convex, and thus, locally, good spectral gap
lower bounds can be obtained (see Lemma~\ref{le.Hess} and 
Proposition~\ref{BLlemma} below). 

\

\noindent
The rest of the section is organized as follows. In Subsection~\ref{preliminaryresults}
we make precise and prove some aforementioned preliminary results which are needed
for the proof of Theorem~\ref{th.CFKS}. In Subsection~\ref{proofCFKS}
we introduce a suitable quadratic partition of unity of $\mb R^N$
and give the proof of Theorem~\ref{th.CFKS}.

\subsection{Preliminary estimates}\label{preliminaryresults}

In the sequel we write $y$ and $z$ to denote generic elements of $\mc C$ and $\mc C^{\perp}$ 
respectively and $dy, dz$ for the Lebesgue measures on $\mc C$ and $\mc C^{\perp}$.
We recall also the notation defined in~\eqref{conditionednotation}.

\

\noindent
The combination of the following two lemmata allows to control the quadratic form 
$\lp \Delta_{f,h} \psi, \psi\rp_{L^2(\mb R^N)}$ away from the diagonal $\mc C$.

\begin{lemma}[Poincar\'{e} inequality for fixed $y\in \mc C$]\label{lemmaconditionalpoincare}
The following inequality holds true for every $h>0, N\in \mb N$, $\psi\in\mc C^\infty_{{\rm c}}(\mb R^N)$ and $y\in \mc C$: 
\begin{equation}\label{condpoincare}    \lp \,  \psi_y ,   \,  (\Delta^{\mc C^{\perp}}_{f,h}   \psi)_y   \, \rp_{L^2(\mc C^{\perp})}   \    \geq    \    
h \   (\mu-1)    \   \Big(    \   \|\psi_y\|^2_{L^2(\mc C^{\perp})}    \   -  \    
\lp  \, \psi_y , \,  E_y \,  \rp_{L^2(\mc C^{\perp})}^2     \    \Big)    \ ,  
\end{equation}
where $E:  \mb R^N \ra  \mb R$ is given by
\begin{equation}\label{conditionalgroundstate}
  E(x)     \    :=  \     \frac{e^{-\frac{f(x)}{h} }}{\big( \  \int_{\mc C^{\perp}} e^{-2\frac{f(Px+z)}{h}}  \ dz \ \big)^{\frac 12}   }
 \   .     
 \end{equation}
 \end{lemma}
\begin{proof}
Note that~\eqref{condpoincare} is equivalent to the inequality
\begin{equation*}
 \int_{\mc C^{\perp}} \!\!  \|\nabla^{\mc C^\perp}\! (E^{-1} \psi)_y\|^2   E^2_y    dz   
    \    \geq    \    
   \frac{\mu-1 }{h}   \   \Big(    \int_{\mc C^{\perp}}  \!\!
  (E^{-1} \psi)_y^2   E^2_y  dz 
    \   -  \ 
\big(\!\!\int_{\mc C^{\perp}} \!(E^{-1}\psi)_y    E_y^2  dz\big)^2     \Big)
\end{equation*}
which  follows from the uniform convexity
estimate
 $$\frac1h\hess^{\mc C^\perp}2 f (x)\ =\  
\frac1h\hess^{\mc C^\perp} V(\sqrt N x)\ \geq
\ \frac{\mu-1 }{h}
$$
implied by~\eqref{uniformconvexity}, and 
standard criteria for the spectral gap of strictly log-concave measures
(see for example \cite[Corollary 11.4]{Le} or the already used Bakry-\'{E}mery criterion of Proposition~\ref{BakryEmery} which
gives an even stronger result). 
\end{proof}

\noindent
Note that by integrating the relation~\eqref{condpoincare}
of Lemma~\ref{lemmaconditionalpoincare} in $y$ and using the Cauchy-Schwarz inequality, 
we get that for every $h>0,\  N\in \mb N$ and $\psi\in C^\infty_c(\mb R^N)$, 
\begin{equation} \label{Poincare+CS}  
   \lp \,  \psi ,   \,  \Delta^{\mc C^\perp}_{f,h}   \psi   \, \rp_{L^2(\mb R^N)}   \    \geq    \    
h    (\mu-1)  \,    \|\psi\|^2_{L^2(\mb R^N)}    \Big(       1       -      
\sup_{y\in \supp \psi}   \int_{\supp \psi_y}   
E^2_{y}(z)    \  dz      \Big)      . 
\end{equation}
In order to fully exploit Estimate~\eqref{Poincare+CS}, we need a control
on the integral appearing on its right hand side when $\psi$
is localised far from the diagonal. The following rough tail estimate
will be enough for our purposes.

\begin{lemma}[Concentration Lemma]
\label{le.concent}

\

\noindent
Let $h$, $R_{h}$ and $\rho$ be three positive numbers. Then there exist $h_0>0$ and $\gamma>0$
such that for every $h\in (0, h_0]$ and  $N\in \mb N$,  
the function $E$ defined in 
\eqref{conditionalgroundstate} satisfies
\begin{equation}\label{concentrationbounded}   \sup_{\|y\|\leq \rho} \int_{\{\|z\|\geq R_h\}}      E^2_y(z)    \    dz       \     \leq    \    
\min \{\gamma  \, e^{- \frac{R_h^2}{\gamma h}} \,  ,   \, 1\}      \  ,       
\end{equation}
and 
\begin{equation}\label{concentrationunbounded}    \sup_{y\in \mc C} \int_{\{\|z\|\geq R_h\}}      E^2_y(z)    \    dz       \     \leq    \    
\min\{e^{\gamma N}  \, e^{- \frac{R_h^2}{\gamma h}}   \,  ,   \, 1\}    \  .       
\end{equation}
\end{lemma}

\begin{proof} For every $\tau\in \mb R^+$, $h>0$, $N\in \mb N$
and $y\in \mc C$,  we have the upper bound
\begin{gather}\label{exponentialMarkov}
\int_{\|z\| \geq R_h} E_y^2(z)    \ dz     \   \leq    \    e^{-\tau\frac{ R^2_h}{h}}  \  
\int_{\mc C^{\perp}} e^{\tau \frac{\|z\|^2}{h}}  \   E_y^2(z)    \ dz 
     \  . 
\end{gather}
\noindent
To estimate the integral on the right hand side of~\eqref{exponentialMarkov},
we shall use the following two bounds on $f$:
\begin{equation}\label{upperBf}
2f(x)    \ -\ 2f(Px) \    \leq\     
 \frac {3N}{4}     \|P^{\perp} x\|_4^4   
\ +\    
\frac 12\lp \, P^{\perp}x\,, \big(K-1  +   4\|Px\|^2 \big) \, P^{\perp}x \, \rp    
      \  ,
\end{equation}
and
\begin{equation}\label{lowerBf}      2f(x)       \ -\ 2f(Px)  \   \geq\   \frac 12\lp \, P^{\perp}x\,, (K-1  +   \|Px\|^2) \, P^{\perp}x \,\rp  
 \   .  
\end{equation}
Estimate~\eqref{lowerBf} follows immediately from 
the definition~\eqref{definitionfWitten} of $f$ and from the inequalities
\[    \frac N4 \|x\|_4^4  \   \geq     \ \frac 14 \|x\|^4   \ =   \  \frac 14 (\|Px\|^2 +  \|P^{\perp}x\|^2 )^2   \   \geq     \   \frac 12 \|Px\|^2 \, \|P^{\perp}x\|^2  \ +\ 
\frac14\|Px\|^4  \]
together with $\|Px\|^4=N\|Px\|^4_{4}=N^2\overline x^4$.
To see~\eqref{upperBf},  note first that from the definition of $f$,\begin{eqnarray*}
 2f(x)  \ -\ 2f(Px)   \  =   \    
  \frac {N}{4}     \|P^{\perp} x\|_4^4    & + &   N \overline x \sum_{k=1}^N (P^{\perp} x)_k^3 \\  
\quad && +     \ 
\frac 12\lp \, P^{\perp}x, \big(K-1  +   3\|Px\|^2 \big) \, P^{\perp}x \, \rp         \  , 
\end{eqnarray*}
so~\eqref{upperBf} is a consequence of the elementary inequalities 
\begin{eqnarray*}  \big|N     \overline x \sum_{k=1}^N (P^{\perp} x)_k^3    \big|  \leq
\sqrt N \|Px\|  \|P^{\perp}x\|_4^2     \|P^{\perp} x\| 
   \leq
\frac12 \|Px\|^2  \|P^{\perp}x\|^2     +      \frac N2 \|P^{\perp}x\|_4^4    \,  .
\end{eqnarray*}
From~\eqref{exponentialMarkov}, together with~\eqref{upperBf},~\eqref{lowerBf}
and computations of Gaussian integrals, 
we obtain for every $h>0$, $N\in \mb N$ and for every $\tau\in (0,\frac{\mu- 1}{2})$, 
\begin{equation}\label{Thetaepsilon}
\int_{\|z\| \geq R_h} E_y^2(z)    \ dz     \   \leq    \    e^{-\tau \frac{R^2_h}{h}}  \  
   \frac {\Theta(\tau, N, y)}{1- \epsilon(h,N,y)}     \   ,   
\end{equation}
where
\begin{equation}\label{pxdeterminants}
\Theta (\tau, N, y)    \   :=   \   \Big( \frac{\det(K-1 + 4\|y\|^2)}{\det (K-1 +\|y\|^2 - 2\tau)}   \Big)^{\frac 12}    \ ,    
\end{equation}
and 
\begin{gather*}
 \epsilon(h,N, y)    \  :   =   \  \frac{\big(  \det (K-1 + 4\|y\|^2)  \big)^{\frac 12}}{(2\pi h)^{\frac N2}}    \int_{\mc C^{\perp}} \!\!  \big( 1 - e^{-\frac{3N\|z\|_4^4}{4h}}\big)    e^{-\frac{\lp z, (K-1+4\|y\|^2)z\rp}{2h}} dz    .   
\end{gather*}
As in the proof of Proposition~\ref{quasimodal1}, we use the simple estimate
\begin{gather*}
 |\epsilon(h,N, y)|    \  \leq    \  \frac{\big(\det (K-1 + 4\|y\|^2)\big)^\frac12}{(2\pi h)^\frac N2}
   \frac{1}{4h} \int_{\mc C^{\perp}}   3N\|z\|_4^4  \   e^{-\frac{\lp z, (K-1+4\|y\|^2)z\rp}{2h}}    \ dz     \  ,   
\end{gather*}
and conclude, by applying a straightforward modification of Lemma~\ref{momentsgaussian}, that there exists a constant $C>0$ such that
for every $N\in \mb N$, $h\in (0,1]$ and $y\in \mb R^N$, 
\begin{equation}\label{epsilonhNy}
 |\epsilon(h,N, y)|    \  \leq    \    C\   h     \  .   
\end{equation}
In order to control $\Theta(\tau, N, y)$, we fix $\tau\in (0, \frac38(\mu-1))$
so that $\Theta(\tau, N, y)$ increases with $\|y\|$
(for any fixed $N$),
and  observe that, 
arguing as in the proof of Lemma~\ref{lemmaconvergenceratiodet}, 
for every $\rho>0$ there exists
a constant $C>0$ such that
\begin{equation}\label{Thetabounded}
  \forall\,N\,\in\,\mathbb N\ ,\quad        \sup_{\|y\|\leq \rho} \Theta(\tau, N, y)     \   \leq    \   C     \    .  
\end{equation}
If $y$ is not constrained to a compact set, we get the existence of
a constant $C>0$ such that  for every $N\in \mb N$, $y\in \mb R^N$, 
\begin{equation}\label{Thetaunbounded}  \Theta (\tau, N, y)    \   \leq    \     4^{\frac N2}    \   \leq   e^{CN}    \  .     \end{equation}
Thus, from~\eqref{Thetaepsilon}, taking $h_0$ and $\gamma^{-1}$ sufficiently small, 
one obtains~\eqref{concentrationbounded}  according to~\eqref{epsilonhNy},~\eqref{Thetabounded}
 and one obtains~\eqref{concentrationunbounded}
according to~\eqref{epsilonhNy},~\eqref{Thetaunbounded}. 

\end{proof}

\noindent
The following lemma shows the existence of a suitable neighbourhood of 
the minimum $J_+$ on which $f$ is uniformly convex. 
Note that by symmetry arguments, the analogous statement holds with $J_-$
instead of $J_+$.

\begin{lemma}[Uniform Convexity around the minima]
\label{le.Hess}
There exist  constants $r, \rho>0$ such that  
 \[    \forall N\in \mb N  \  ,    \   \forall  x\in \Omega_{r}    \  ,   \   \    \      \hess f(x) \     \geq\    \rho         \      \  , \] 
where the set $\Omega_r$ is given by
\[    \Omega_r    \  :  =    \   \big\{   \  x\in \mb R^N  \   :   \  \|Px     -    J_+\| \  \leq \  r  \  , \
\|P^{\perp}x\| \  \leq    \  r    N^{-\frac 14} \  \big\}    \  .  \]
\end{lemma}

\begin{proof}
For $N\in \mb N$ take $x, w\in \mb R^N$ with $\|w\|=1$. Then, recalling~the expression of $f$ given in \eqref{definitionfWitten}, we get 
\begin{equation}\label{hessian}
2 \lp w, \hess f(x) w\rp    \ =  \       \lp w, (K-1)w\rp      \    +   \     3 N   \sum_{k=1}^{N-1} x_k^2 w_k^2      \  .    
\end{equation}
For the first term in~\eqref{hessian}, the discrete 
Poincar\'{e} inequality~\eqref{discretePoincare} gives 
with  $\rho := \frac 14\min\{\mu-1, 1\}$ the lower bound
\begin{gather}\label{discretePoincarehessian} 
 \lp w, (K-1)w\rp      \   \geq   \
    4\rho \|w\|^2  \   -     \      (1+4\rho) \|Pw\|^2   
       \    
 \geq    \   4\rho   \   -     \   2  \|Pw\|^2   
      \    .    \end{gather}
To estimate the second term in~\eqref{hessian},
we use the decomposition $\id = P + P^{\perp}$
and a
straightforward computation yields
\begin{eqnarray} 
 3  N \sum_{k=1}^N x_k^2 w_k^2    &\geq& 
 3   N \sum_{k=1}^N \big(  \overline x^2  +  2 \overline x (P^{\perp}x)_k + (P^{\perp}x)_k^2   \big)   
 \big( \overline w^2  +  2 \overline w (P^{\perp}w)_k   \big)    \nonumber  \\ 
       &\geq&  3 \|Px\|^2  \|Pw\|^2     -         12 \|Px\|    \|P^{\perp}x\|     -    6
\sqrt N    \|P^{\perp}x\|^2        \,   .\label{quartichessian} 
 \end{eqnarray}
Note that by the triangular inequality, 
we have  the two uniform bounds for every $r>0$,
\begin{equation*}
\forall N\in \mb N   ,  \ \forall \,  x\in \Omega_r      \  ,   \   \   \      \    \    
  1-r    \     \leq    \  \|Px\|  \   \leq   \   1 +r          \ . 
\end{equation*}  
Thus, estimate~\eqref{quartichessian} gives for every $r>0$, $N\in \mb N$
and $x\in \Omega_r$, 
\begin{gather*}
3  N \sum_{k=1}^N x_k^2 w_k^2    \  \geq     \    3(1-r)^2   \|Pw\|^2    \   -    \   12 (1+r) rN^{-\frac14}
\  -      \  6   r^2      \   .
\end{gather*}
Taking $r>0$ sufficiently small we get 
$3 N   \sum_k x_k^2 w_k^2    \geq         2  \|Pw\|^2       -    2 \rho  $, 
which together with~\eqref{hessian} and~\eqref{discretePoincarehessian}
finishes the proof. 

\end{proof}

\noindent
The preceding Lemma~\ref{le.Hess} is used to establish the 
localized spectral gap estimate of Proposition~\ref{BLlemma} below.
As in the proof
of Lemma~\ref{lemmaconditionalpoincare}, we argue by means of standard results for strictly log-concave measures. To reduce to the standard situation, we use the following 
general result on convex extensions, for which we provide a proof for the sake 
of completeness. 

\begin{lemma}\label{lemmaextension}
Fix $d\in \mb N$. Let $\vp\in \mc C^\infty(\mb R^d)$ and $A$ be
a compact and convex subset of $\mb R^d$ such that
\[\exists \,\varepsilon>0\ ,\  \exists\,C>0\quad \text{s.t.}\quad\Hess \vp\,\geq\, C
\  \text{on}\   A_{\varepsilon}:=\left\{x+y\ ;\ x\in A\,,\ \|y\|\leq\varepsilon\right\}   \  .     \]
Then there exists $\tilde \vp\in \mc C^\infty(\mb R^d)$ such that 
\[     \forall x\in A     \ ,    \  \   \   \tilde \vp (x)   =         \vp(x)     \ 
\  \    \text{ and  }     \  \   \    \forall x\in \mb R^d     \ ,    \  \   \     \Hess \tilde\vp(x)   \     \geq   \     C       \  . \]
\end{lemma}
\begin{proof}
The proof consists in smoothly cutting $\vp$  outside $A$ and adding
a function $g$ vanishing on $A$ and sufficiently convex outside $A$. 
To easily construct such a function $g$ it is convenient to reduce to radial cut-off's as follows (see \cite{Yan}): first, since 
$A$ is convex and compact,  
$A_{\varepsilon}\setminus \mathring{A_{\frac\varepsilon2}}$
is compact and there exist $\ell\in\mathbb N$,  $(x_{i})_{i\in\{1,\dots,\ell\}}\subset (\mathbb R^N)^\ell$ and $(r_{i})_{i\in\{1,\dots,\ell\}}\subset (0, \infty)^\ell$ such that, 
denoting by $\overline B (x_{i},r_{i})$ the closed ball 
of radius $r_i$ centered at $x_i$, 
\begin{equation}\label{convexsets}
A\ \subset\ \cap_{i=1}^{\ell}\overline B (x_{i},r_{i})
\ \subset \ \mathring{A}_{\frac\varepsilon2}\ \subset\ A_{\varepsilon}\,     \   . 
\end{equation}   
We shall consider  $\theta\in C^{\infty}(\mb R)$ defined by
$$
\theta(t)  :=\left\{\begin{array}{lcr}
0\quad &\text{if}&\quad t\in(-\infty,1]\ \, \\
t^2e^{-\frac{1}{t-1}}\quad&\text{if}&
\quad  t\in(1,+\infty)\,.
\end{array}\right.
$$ 
This function 
is  strictly increasing on 
$(1,+\infty)$, as well as
$t\mapsto\frac{\theta'(t)}{t}$,
since we have
for any $t>1$,
$$
\theta'(t)\ =\  \left(2t+\frac{t^2}{(t-1)^2}\right)e^{-\frac{1}{t-1}}
\quad\text{and}\quad
\left(\frac{\theta'(t)}{t}\right)'\ =\ \frac{t^2-3t+3}{(t-1)^4}
e^{-\frac{1}{t-1}}\,.
$$
Moreover notice that $\Hess\big(x\mapsto \theta(\|x\|)\big)$ is given
by
$$
\Hess\big(x\mapsto \theta(\|x\|)\big)
\,=\, \frac{\theta'(\|x\|)}{\|x\|}\mathop{Id}\, +\, 
\frac{1}{\|x\|}\left(\frac{\theta'(t)}{t}\right)'_{t=\|x\|}\big(x_{i}x_{j}\big)_{1\leq i,j\leq N}\,,
$$
and so is positive definite for $\|x\|>1$ (and zero for $\|x\|\leq 1$).
We then define 
\[
g(x)   \   :=    \   \sum_{i=1}^{\ell} \theta 
\left( \frac{\|x-x_{i}\|}{r_{i}} \right)   \,.
\]
Note that $g$ is smooth, $\hess g \geq 0$ and that, according to~\eqref{convexsets}, $g\equiv 0$ on $A$ and $\hess g > 0 $ on the complementary of $\mathring{A}_{\frac\varepsilon2}$. Finally we define the following extension of $\vp|_{A}$,  
\[    \tilde \vp      \  :   =   \       \chi    \,   \vp      \  +   \    \alpha    \,  g          \      ,   \]
where  $\chi\in \mc C_{{\rm c}}^\infty(\mb R^N)$ satisfies
 $\chi \equiv1$ on $A_{\frac\varepsilon2}$ and $\supp \chi\subset A_{\varepsilon}$,
and  $\alpha>0$  is chosen large enough so that
$\hess \tilde \vp  \geq  C$. 
This is indeed possible since $\Hess \tilde \vp \geq \Hess \vp \geq C$ on $\mathring{A}_{\frac\varepsilon2}$,  $\Hess \tilde \vp 
 = \alpha\,\Hess g$ on $\mathbb R^N\setminus A_{\varepsilon}$,
 $\Hess \tilde \vp=\alpha\,\Hess g+\Hess (\chi\,f)$ on $A_{\varepsilon}\setminus
 \mathring{A}_{\frac\varepsilon2}$
 and $\min\{\Hess g(x)\,,\ x\in \mathbb R^N\setminus \mathring{A}_{\frac\varepsilon 2}\}>0$.  
\end{proof}

\noindent
As a corollary of Lemmata~\ref{le.Hess} and~\ref{lemmaextension}, 
the following spectral gap estimate for a suitably localized problem around the minimum $J_+$ holds true.
Note that the analogous version around the other minimum $J_-$ holds true
by symmetry. 
\begin{proposition}\label{BLlemma}
Let $r>0$ and $\vp\in \mc C_{{\rm c}}^\infty(\mb R^N)$ such that 
\begin{equation}\label{BLcond}    \supp \vp    \   \subset    \       
\Omega_r   :=      \big\{   \  x\in \mb R^N  \   :   \  \|Px     -    J_+\| \  \leq \  r \  , \
\|P^{\perp}x\| \  \leq    \  r    N^{-\frac 14} \  \big\}    \   .      
\end{equation}
If $r$ is suffiently small, then there exists a constant $\rho>0$
such that  for all $h>0$ and $ N\in \mb N$,   there exists $\mc E^+_{h}\in L^2(\mb R^N)$
such that 
\begin{equation}\label{BLlocal}       \lp \, \vp  ,   \,     \Delta_{f,h} \vp \, \rp_{L^2(\mb R^N)}   \   \geq    \   \rho   \,  h \
   \Big(    \    \|\vp\|^2_{L^2(\mb R^N)}     \  -    \   \lp \vp, \mc E^+_{h}  \rp^2_{L^2(\mb R^N)}  \  \Big) \ .   \end{equation}
\end{proposition}

\begin{proof}
According to Lemma~\ref{le.Hess}, by taking $r>0$ sufficiently small
there exists $\rho>0$ such that  
\[    \forall N\in \mb N  \  ,    \   \forall  x\in \Omega_{2r}    \  ,   \   \    \      \hess f(x) \     \geq\    \rho         \      \  . \] 
By Lemma~\ref{lemmaextension}, there exists for each $N\in \mb N$ 
a function $\tilde f\in \mc C^\infty(\mb R^N)$ such that $\tilde f|_{\Omega_r}  \equiv
f|_{\Omega_r}   $ and 
\begin{equation}\label{extendedfestimate}    \forall N\in \mb N  \  ,    \   \forall  x\in \mb R^N    \  ,   \   \    \      \hess \tilde f(x) \     \geq\    \rho         \      \  . 
\end{equation}
As in the proof of Lemma~\ref{lemmaconditionalpoincare}, by standard results for the spectral gap of strictly log-concave measures
(see for example \cite[Corollary 11.4]{Le}),
Property~\eqref{extendedfestimate} implies~\eqref{BLlocal},
with the differential operator $\Delta_{\tilde f,h}$ instead of $\Delta_{ f,h}$ and with 
\[   \mc E^+_{h}   (x)    \  :=    \      \frac{e^{-\frac{\tilde f(x)}{h}}}{\int_{\mb R^N}e^{-2\frac{\tilde f(x)}{h}} \  dx}   \ . \]
Noting that $  \lp \, \vp  ,   \,     \Delta_{\tilde f,h} \vp \, \rp_{L^2(\mb R^N)}  =
  \lp \, \vp  ,   \,     \Delta_{f,h} \vp \, \rp_{L^2(\mb R^N)}$
  for any smooth $\varphi$ with support in $\Omega_{r}$, which follows from 
$\tilde f|_{\Omega_r}  \equiv
f|_{\Omega_r} $,  finishes the proof. 

\end{proof}

\subsection{Proof of Theorem~\ref{th.CFKS}}\label{proofCFKS}

\noindent
We fix from the outset $C>0$ and $\alpha\in (0,\frac 34)$. As we already mentioned,
it is sufficient to prove~\eqref{goal}. For this we introduce as follows a quadratic partition of unity $\{\eta_k\}$ depending
on the given $\alpha$ and on a parameter $r>0$,
independent of $N$ and $h$, which will be
chosen sufficiently small so that the estimates required   for the proof 
hold true. 

\noindent
We start with $\theta := \theta_{r}\in \mathcal C_{{\rm c}}^{\infty}(\mb R;[0,1])$
such that $\theta(x) =\theta(-x)$, $\theta\equiv1$ in $[-r,r]$ and $\theta \equiv 0$
in $[2r,+\infty)$ and define $\kappa_{\min}, 
\kappa_0, \kappa_\infty:\mb R^N \ra [0,1]$ by setting
\begin{equation}\label{defkappamin}  \kappa_{\min}(x)   \   :=    \   \theta (\|Px - J_{+}\|)   \   +   \  
\theta (\|Px - J_{-}\|)          \   ,    
\end{equation}
\begin{equation}\label{defkappa0}
    \kappa_{0}  \    :
=    \   \big(  1- \kappa_{\min}^2 \big)^{\frac{1}{2}} \ \textbf{1}_{\{\|Px\|\leq 1\}}    
\    \   ,      \     \   
\kappa_{\infty}:=\left(1-\kappa_{\min}^2\right)^\frac12\textbf{1}_{\{\|Px\|\geq 1\}}    \  . 
 \end{equation}
Moreover we define for $p\in\{4,6\}$ the functions
$\chi_{0,p}, \chi_{\infty, p}:\mb R^N \ra [0,1]$  as
\begin{equation}  \label{defchi} \chi_{0, p}(x):=\theta(   h^{-\alpha/p} \|P^\perp x\|)        \   \    \   \text{and}
\    \   \  
  \chi_{\infty,p}     : =     \big( 1  -    \chi^2_{0,p}\big)^\frac 12     \     .      
 \end{equation}
Note that the $\chi$'s depend on $h$, while the $\kappa$'s do not. 
Note also that 
\begin{equation}  \label{partitionunity}\big(   \   \kappa_{\min}^2 \ +    \  \kappa_{0}^2 
\  \big)   
   \
\big(  \    \chi^2_{0,4}    \    +   \   \chi^2_{\infty,4}     \   \big)      \     +    
   \   \kappa_{\infty}^2   \  
\big(  \    \chi^2_{0,6}    \    +   \   \chi^2_{\infty,6}     \   \big)
  \    \equiv   \  1      \   .
\end{equation}
We shall consider in the sequel the partition of unity $\{\eta_1, \dots, \eta_6\}$, 
where $\eta^2_k$ is given by one of the six products $\kappa_j^2 \chi_{j'}^2$ 
appearing when multiplying out the left hand side of~\eqref{partitionunity}. 
Observe that for each $k\in\{1,\dots, 6\}$,
\[  \forall N\in \mb N  \, ,    \,   \forall x\in \mb R^N  \,  ,     \   \  \        h^2|\nabla \eta_k(x)|^2    \     \lesssim     \   h^{2-\frac \alpha 2}  \     . \]
Thus the IMS localization formula of Proposition~\ref{pr.IMS} implies
that for every $\psi\in \mc C_{{\rm c}}^\infty(\mb R^N)$ and $h$ sufficiently small, 
\begin{equation}
\label{eq.IMS2}
\lp \psi , \Delta_{f,h} \psi   \rp_{L^2(\mb R^N)}\ +\ 
h^{2-\frac \alpha 2} 
   \     \gtrsim     \  
\sum_{k=1}^{6}    \  \lp  \,  \eta_{k} \psi,  \,  \Delta_{f,h}   (  \eta_{k} \psi)   \,    \rp_{L^2(\mb R^N)}
 \;    .  
\end{equation}

\noindent
Here and in the sequel we shall use for short the notation $\gtrsim$
and $\lesssim$ to denote inequalities which hold true up to multiplication 
of (say) the right hand side by a positive constant which is independent of $h$ and $N$.

\

\noindent
In the rest of the proof, we fix a $\psi\in \mc C_{{\rm c}}^\infty(\mb R^N)$
and discuss separately the addends on the right hand side
of~\eqref{eq.IMS2}.   

\

\

\

\noindent
{\bf Analysis around the diagonal}\\[0.1cm]
\noindent
{\it a)  Analysis on $\supp \big(     
 \kappa_{\min}     \,  \chi_{0,4} \big)$: 
  }
According to the definitions given in~\eqref{defkappamin} and~\eqref{defchi}, we have
for every $h>0$ and $N\in \mb N$ such that  $N\leq Ch^{-\alpha}$, 
\begin{equation*}
 \supp    (   \, \psi \,  \kappa_{\min}     \,  \chi_{0,4}   \, )         \    \subset   \  
 \Omega_{+,r}    \    \cup    \    \Omega_{-,r}           \          ,      
\end{equation*}
\[  \text{ where }  \   \   \     \Omega_{\pm,r}     \    :=   \   \big\{   x\in \mb R^N     :    \|Px     \pm    J_+\| \  \leq \  2 r \  , \
\|P^{\perp}x\|  \leq    2 r C^{\frac 14}  N^{-\frac 14}  \  \big\}  \  .    \]
Then it follows from Proposition~\ref{BLlemma} (and its analogous version
around $J_-$) that, chosing $r$ sufficiently small,   for all $h>0$ and $ N\leq Ch^{-\alpha}$  there exist $\mc E^+_{h}, \mc E^-_{h}\in L^2(\mb R^N)$
such that, denoting for short $\vp:= \kappa_{\min}  \chi_{0,4}\,\psi  $, 
\begin{equation}      
 \lp \, \vp   ,   \,     \Delta_{f,h} \,\vp \, \rp_{L^2(\mb R^N)}     \gtrsim   \label{estimateIa} 
       h 
   \left(   \| \vp \|^2_{L^2(\mb R^N)}       -       \lp 
   \vp, \mc E^+_{h} \rp^2_{L^2(\mb R^N)}
      -       \lp \vp, \mc E^-_{h}  \rp^2_{L^2(\mb R^N)} \right)  .  
   \end{equation}
 
\

\noindent
{\it b) Analysis on $\supp \big(     
 \kappa_{0}^2   \,     \chi_{0,4}^2   +    
 \kappa_{\infty}^2     \,   \chi_{0,6}^2     \big)$: 
  }
Here we shall use that, in the sense of quadratic forms, 
\begin{equation} \label{arounddiagonal} \Delta_{f,h}    \    \geq      \    \Delta^{\mc C}_{f,h}  \     \geq    \   
|\nabla^{\mc C} f|^2     \  -     \     h  \,  \Delta^{\mc C}  f     \      .  
\end{equation}
The final estimate~\eqref{estimateIb} given below follows then by elementary inequalities
which we spell out for completeness. 
Note first that the  
 definitions~\eqref{defkappa0} and~\eqref{defchi}  imply  in particular that
for every $h>0$ and $N\in \mb N$, 
\begin{equation}\label{suppIb}  \supp \big(     
 \kappa_{0}^2   \,     \chi_{0,4}^2   +    
 \kappa_{\infty}^2     \,   \chi_{0,6}^2     \big)    \     \subset   \    \Omega_0  \    \cup    \   
   \Omega  \  , \end{equation}
    where $\Omega$ and $\Omega_0$ are defined as
\begin{equation*} 
  \Omega    \  :=   \    \big\{   x\in \mb R^N     :    \|Px     -    J_{\pm}\|  \geq    r \  , \
   \|Px    \|   \geq    r  \  ,   \  
\|P^{\perp}x\|  \leq     2r  h^{\frac \alpha 6}  \  \big\}  \  ,
\end{equation*}
and
\begin{equation*}    
\Omega_0    \  :=   \    \big\{   x\in \mb R^N     :    \|Px\| \leq   r \  , \
\|P^{\perp}x\|  \leq     2r  h^{\frac \alpha 6}  \  \big\}    \   . 
\end{equation*}
On $\Omega_0$ one can immediately give a lower bound 
for the right hand side in~\eqref{arounddiagonal}. Indeed, 
 chosing $r$ sufficiently small, we have from \eqref{eq.laplx0}
 for
 every $x\in \Omega_0$,  $h\in (0,1]$, and 
 $N\in \mb N$,
 \begin{equation}\label{omega0estimate}      |\nabla^{\mc C} f|^2   \    -   \       h  \,    \Delta^{\mc C}  f    \     \geq  \    
  -   \       h  \,    \Delta^{\mc C}  f    \  =   \   
    \frac h2 \,  ( 1   -   3\hat x_{0}^2-3 \|P^\perp x\|^2)\  \geq\    \frac h4    \   .   
    \end{equation}
To deal with $\Omega$, we develop the expression of $\nabla^{\mc C} f(x)$ given in \eqref{eq.nablax0},
\[ \nabla^{\mc C} f(x)    \   =   \  
 \frac 12\hat x^3_0   \   -    \    \frac 12 \hat x_0 \   + \   \frac32 \hat x_0 \|P^{\perp}x\|^2
 \   +   \    \frac 12 \sqrt N \sum_{k=1}^N (P^{\perp}x)_k^3     \  ,    \]
 from which we get, using that for all $h>0$ we have $\|P^{\perp}x\|  \leq     2r  h^{\frac \alpha 6}$,
\[    \forall N\leq C h^{-\alpha}   \   ,    \  \   \  |\sqrt N \sum_{k=1}^N (P^{\perp}x)_k^3|   \   \leq    \    8\sqrt C r^3      \  .      \]
Then, we obtain for sufficiently small $h$ the lower bound
\begin{equation}\label{estnabla}    |\nabla^{\mc C} f(x)|    \ 
 \geq    \    \left|\frac 12\hat x^3_0 \  -  \  \frac12\hat x_0 \  
 + \   \frac32 \hat x_0 \|P^{\perp}x\|^2\right|
 \ -
     \   4\sqrt C r^3   \     ,
\end{equation}
from where it follows, choosing $r$ sufficiently small, that
\begin{equation}\label{discussionpotential}\begin{cases}
| \nabla^{\mc C} f(x)|   \gtrsim  1     &    \text{  for }   x\in \Omega    \ \text{ s.t. }  \,    r   \leq  \|Px\|   \leq 1-r   \\
 | \nabla^{\mc C} f(x)|   \gtrsim  \|Px\|^3  &    \text{  for }  x\in \Omega   \ \text{ s.t. }   \, 1+r  \leq \|Px\|   \ .
\end{cases} 
\end{equation}
Combining~\eqref{discussionpotential} with the estimate 
$$\forall\,x\,\in\,\Omega\,,\  h \ |\Delta^{\mc C} f(x)|     \ =\ \frac{h}{2}\left|\,1-3\|Px\|^2-3\|P^\perp x\|^2\,\right|
\ 
\leq    \   \frac h2\max\{1\,,\,3\|Px\|^2 \} \, $$ 
valid  for $h$ sufficiently small, we
finally  get the existence of $h_0>0$ such that  
 \begin{equation}\label{omegaestimate}  \forall x\in \Omega \  ,  \forall h\in (0,h_0]  
 \  ,  \forall N\leq C h^{-\alpha}    \    ,  \    \   \  \      \       |\nabla^{\mc C} f|^2   \    -   \       h  \,    \Delta^{\mc C}  f    \     \gtrsim   \      1
    \   .   
    \end{equation}
Summing up this part, 
setting for short $\vp:=   (   \,    
 \kappa_{0}^2   \,     \chi_{0,4}^2   +    
 \kappa_{\infty}^2     \,   \chi_{0,6}^2  \, )^\frac12   \, \psi$, it 
 follows from~\eqref{arounddiagonal}--\eqref{omega0estimate} and~\eqref{omegaestimate} that 
 there exist $h_0>0$ 
and $r$ sufficiently small such that, for every $h\in (0,h_0]$ and 
every $N\in \mb N$
satisfying $N \leq C h^{-\alpha}$, 
\begin{equation}\label{estimateIb}
 \lp \,    
 \vp  ,   \,     \Delta_{f,h} \,\vp   \, \rp_{L^2(\mb R^N)}   \   \gtrsim   \ \|\vp\|^2_{L^2(\mathbb R^N)}   \ .
\end{equation}

\

\noindent
{\bf Analysis away from the diagonal}\\[0.1cm]
\noindent 
Here it is convenient to work with  $\Delta^{\mc C^{\perp}}_{f,h}$,
which is sufficient due to the inequality $\Delta_{f,h}       \geq     
\Delta^{\mc C^{\perp}}_{f,h}$. 

\

\noindent
{\it a) Analysis on $\supp \big(     
 (\kappa^2_{\min}+\kappa^2_{0} )    \  \chi_{\infty,4} \big):$ 
  }
Let for short $\vp :=     
  (\kappa^2_{\min}+\kappa^2_{0} )^\frac12  \chi_{\infty,4}  \  \psi$. 
Note that by the definitions~\eqref{defkappamin}, \eqref{defkappa0},
and \eqref{defchi}, we have
\begin{equation}\label{suppoffIIa} 
  \supp \vp   \  \subset     \   \{x\in \mb R^N: \|Px\| \leq 1+2r 
\   ,  \   \|P^{\perp}x\|\geq r h^{\frac \alpha 4}\}     \  .    
\end{equation}
It follows from the Poincar\'{e} inequality~\eqref{Poincare+CS}, 
the concentration estimate~\eqref{concentrationbounded} of Lemma~\ref{le.concent}
and~\eqref{suppoffIIa} that there exists a constant $\gamma>0$ such that
\begin{equation}\label{estimateoffa}  \lp \,  \vp ,   \,  \Delta_{f,h}   \vp   \, \rp_{L^2(\mb R^N)}   \    \geq    \    
h \   (\mu-1)    \     \|\vp\|^2_{L^2(\mb R^N)}     \   \Big(     1     -      \gamma e^{-  \frac{r^{2} h^{\frac \alpha 2}}{\gamma h}  }   \Big)         \ . 
\end{equation}
Since $0<\alpha<\frac34$, estimate~\eqref{estimateoffa}
implies that there exists $h_0>0$ 
such that
\begin{equation}\label{estimateIIa}
\forall h\in (0,h_0] \,,\ \forall  N\in \mb N    \    ,  
\   \   \    \lp \,  \vp ,   \,  \Delta_{f,h}   \vp   \, \rp_{L^2(\mb R^N)}   \      \geq    \   h \   \frac{(\mu-1)}{2}    \    \|\vp\|^2_{L^2(\mb R^N)}   \ .
\end{equation}

\

\noindent
{\it b) Analysis on $\supp \big (    \kappa_{\infty}\  \chi_{\infty,6} \big)     $:}
Let $\vp :=   \kappa_{\infty}\  \chi_{\infty,6}  \ \psi$. 
By the definitions~\eqref{defkappa0} and \eqref{defchi}, we have
\begin{equation}\label{suppoffa}   \supp \vp   \  \subset     \   \{x\in \mb R^N: 
\    \|P^{\perp}x\|\geq r h^{\frac \alpha 6}\}     \  .    
\end{equation}
As for the point a) above, we use the Poincar\'{e} inequality~\eqref{Poincare+CS} but
we can only use here the 
concentration estimate~\eqref{concentrationunbounded} of Lemma~\ref{le.concent} since $\supp\psi$ is arbitrary. 
This leads to the existence of  a constant $\gamma>0$ such that
\begin{equation}\label{estimateoffb}  \lp \,  \vp ,   \,  \Delta_{f,h}   \vp   \, \rp_{L^2(\mb R^N)}   \    \geq    \    
h  (\mu-1)    \     \|\vp\|^2_{L^2(\mb R^N)}     \   \Big(     1     -       e^{\gamma N -  \frac{r^{2} h^{\frac \alpha 3}}{\gamma h}  }   \Big)         \ . 
\end{equation}
Since $0<\alpha<\frac34$, estimate~\eqref{estimateoffb}
implies that there exists $h_0>0$
such that
\begin{equation}\label{estimateIIb}
\forall h\in (0,h_0] \,,\ \forall N\leq Ch^{-\alpha}     \,    ,  
\        {\lp \,  \vp ,   \,  \Delta_{f,h}   \vp   \, \rp}_{L^2(\mb R^N)}   \      \geq    \   h  \frac{(\mu-1)}{2}     \|\vp\|^2_{L^2(\mb R^N)}   .
\end{equation}

\

\noindent
{\bf End of the proof}\\[0.1cm] 
\noindent
Chosing the parameter $r>0$ of the partition of unity
$\{\eta_k\}$
sufficiently small and putting together \eqref{eq.IMS2}, \eqref{estimateIa}, \eqref{estimateIb}, \eqref{estimateIIa}, and \eqref{estimateIIb}, we obtain estimate~\eqref{goal} 
with $E_{\pm}:=  \mc E^{\pm}_{h} \,  \kappa_{\min}     \,  \chi_{0,4}
$.

\qed

\

\noindent
{\bf Acknowledgements}\\
 \noindent
This work was supported by the  NOSEVOL ANR 2011 BS01019 01
and
the European Research Council under 
the European Union's Seventh Framework Programme (FP/2007-2013) / ERC 
Grant Agreement number 614492. The authors' collaboration was initiated thanks to a research visit 
to Paris-Sud University of the first author, then founded by the Department of Mathematics
of the University of Rome La Sapienza, and completed during a ``D\'el\'egation INRIA'' of the second author at CERMICS.
\\[0.1cm] 
Both authors thank Bernard Helffer and Nils Berglund for various discussions, Anton Bovier for pointing to the problem, and Tony Leli\`evre
for offering excellent working conditions at CERMICS.

\end{document}